\documentclass[11pt,reqno]{amsart}

\usepackage{a4wide}
\usepackage[utf8]{inputenc}
\usepackage[T1]{fontenc}
\usepackage{amsmath}
\usepackage{amssymb}
\usepackage{amscd}
\usepackage{amsthm}
\usepackage{proof}
\usepackage{color}
\usepackage{enumerate}
\usepackage{hyperref}

\newtheorem{theorem}{Theorem}[section]
\newtheorem{coro}[theorem]{Corollary}
\newtheorem{lemma}[theorem]{Lemma}
\newtheorem{prop}[theorem]{Proposition}

\newtheorem{definition}[theorem]{Definition}

\theoremstyle{definition}
\newtheorem*{remark}{Remark}

\newcommand{\nts}{\hspace{-0.5pt}}
\newcommand{\dd}{\, \mathrm{d}}

\newcommand{\Cu}{C_{\mathsf{u}}}
\newcommand{\WAP}{\operatorname{W\nts A P}}

\newcommand{\nfunction}{\mathcal{N}}
\newcommand{\overlined}{\overline{d}}

\newcommand{\Average}{\mathrm{A}}
\newcommand{\Freq}{\mathrm{Freq}}
\newcommand{\Bap}{\mathrm{Bap}}
\newcommand{\Eig}{\mathrm{Eig}}

\newcommand{\NN}{\mathbb{N}}
\newcommand{\ZZ}{\mathbb{Z}}
\newcommand{\CC}{\mathbb{C}}
\newcommand{\RR}{\mathbb{R}}
\newcommand{\TT}{\mathbb{T}}

\newcommand{\AAA}{\mathbb A}

\newcommand{\cA}{\mathcal{A}}

\begin{document}

\title{Pure point spectrum for dynamical systems\\ and mean almost periodicity}

\author{Daniel Lenz}
\address{Mathematisches Institut, Friedrich Schiller Universit\"at Jena,
07743 Jena, Germany}
\email{daniel.lenz@uni-jena.de}
\urladdr{http://www.analysis-lenz.uni-jena.de}

\author{Timo Spindeler}
\address{Department of Mathematical and Statistical Sciences, \newline
\hspace*{\parindent}632 CAB, University of Alberta, Edmonton, AB,
T6G 2G1, Canada} \email{spindele@ualberta.ca}

\author{Nicolae Strungaru}
\address{Department of Mathematical Sciences, MacEwan University \\
10700 -- 104 Avenue, Edmonton, AB, T5J 4S2, Canada\\
and \\
Institute of Mathematics ``Simon Stoilow''\\
Bucharest, Romania} \email{strungarun@macewan.ca}
\urladdr{http://academic.macewan.ca/strungarun/}

\begin{abstract} We consider metrizable ergodic   topological dynamical
systems over locally compact, $\sigma$-compact abelian groups. We
study pure point spectrum via suitable notions of almost periodicity
for the points of the dynamical system. More specifically, we
characterize pure point spectrum via mean almost periodicity of
generic points. We then go on and show how  Besicovitch almost
periodic points determine both eigenfunctions and the measure in
this case. After this, we characterize those systems
arising from Weyl almost periodic points and use this to
characterize weak and Bohr almost periodic systems. Finally, we
consider applications to aperiodic order.
\end{abstract}

\maketitle

\tableofcontents

\section*{Introduction}
This article is concerned with dynamical systems with  pure point
spectrum. The dynamical systems in question consist of a continuous
action of a locally compact abelian group on a compact metric space
together with an invariant probability measure on the space. Pure
point spectrum means that there exists an orthonormal basis of
eigenfunctions.  In a sense, such systems are the simplest possible
dynamical systems. Their study is a most basic  ingredient in the
conceptual theory of dynamical systems as witnessed by such
fundamental results as the Halmos--von Neumann theorem or the
Furstenberg structure theorem.  Accordingly, a variety of
characterizations for pure point spectrum has been established over
the decades.

Recent years have brought two new lines of interest in such systems.
One line is given by a series of works which analyze such systems
via (weak) notions of equicontinuity \cite{DG,FGL,GR,GM,GRJY}. The
main thrust of these works is to characterize pure point spectrum as
well as various strengthened versions thereof (see \cite{FGJ,Ver}
for related work as well). The other line comes from the
investigation of aperiodic order. Aperiodic order, also known as
mathematical theory of quasicrystals, has emerged as fruitful field
of (not only)  mathematics  over the last three decades,  see e.g.
\cite{TAO} for a recent monograph and \cite{BaakeMoody00,KLS} for
recent collections of surveys. A key feature of aperiodic order is
the occurrence of (pure) point diffraction. A central result in the
mathematical treatment of aperiodic order gives that pure point
diffraction can be understood as pure point spectrum of suitable
associated dynamical systems. In fact, this result is the outcome of
a cumulative effort of various people over the last decades
\cite{Dwo,LMS,Gou-1,Martin2,BL,LS,LM}.

A common feature of all these works is that their considerations
share a flavor of almost periodicity. On an intuitive level this is
not surprising. After all, pure point spectrum means that all
spectral measures are pure point measures. This in turn is
equivalent to the Fourier transforms of these measures being almost
periodic \cite{MoSt,GLA}. In this way, almost periodicity properties of functions
and their averages come into play in a very natural way.

However, what is lacking is a description of pure point
spectrum via almost periodicity properties of the \textbf{points} of
the dynamical system.  The goal of this article it to provide such a
characterization and to study some of its consequences.

In order to do so, we  introduce for points in dynamical systems the
concepts of mean almost periodicity, Besicovitch almost periodicity,
Weyl almost periodicity, weak almost periodicity and
Bohr almost periodicity,  respectively.

After the discussion of the necessary background in Section
\ref{s:mean}, our first achievement is  Theorem \ref{t:main} in
Section \ref{s:main}. This theorem says that a system has pure point
spectrum if and only if every (or even just one) generic point is
mean almost periodic. If the dynamical system is ergodic this is
then equivalent to almost all points being mean almost periodic.

As for mean almost periodicity of points in turn various
characterizations are discussed in Section \ref{s:main}. One of them
proceeds  via averages of distance between the orbit of the point
and a shifted orbit.  Another one characterizes mean almost
periodicity  of a point via almost periodicity properties of the
sampling of continuous functions along the orbit of this point. In
this way, we have a rather complete and clear picture of the meaning
of mean almost periodic points for general dynamical systems.

This picture ties in with various earlier results.  In the case
where the group is just the integers a related characterization via
sampling of bounded measurable function is given in \cite{BeLo} for
measurable dynamical systems. In the more specific situation of
subshifts over a finite alphabet there is also a characterization of
pure point spectrum via a metric property of points in \cite{Ver}.
Moreover, we point out a companion article \cite{LeSpStr}  dealing
with fundamental issues in aperiodic order via almost periodicity of
measures. As an application, it treats the situation of special
dynamical system viz translation bounded measures dynamical systems.
These systems are particularly relevant to aperiodic order.  In
these systems the points are measures and this allows one to work
with almost periodicity properties of measures. Of course, this
approach it not available for general dynamical systems.

In terms of methods it should  be emphasized that our proof of
Theorem \ref{t:main}  is completely different from the ones given in
\cite{LeSpStr,BeLo,Ver}. It relies on a recent characterization of
pure point spectrum by Bohr almost periodicity of an averaged metric
recently obtained by one of the authors \cite{Len}.

In Section \ref{s:strengthening}, we then introduce  Besicovitch
almost periodic points. While the condition of Besicovitch almost
periodicity is strictly stronger than mean almost periodicity, we
can still show that an ergodic  dynamical system has pure point
spectrum if and only if almost all points are Besicovitch  almost
periodic and this holds if and only if there exists one generic
Besicovitch almost periodic point (Theorem \ref{t:main:bap}).  In
this respect the difference between mean and Besicovitch almost
periodicity  is not too big. The advantage of Besicovitch  almost
periodic points is that they allow for averaging with characters. In
particular, it is possible to compute the eigenfunctions via the
Besicovitch  almost periodic points (Theorem \ref{t:main:bap}). In
fact, the complete spectral theory of the dynamical system can be
directly computed from  any single typical Besicovitch almost
periodic point (Theorem \ref{t:converse}).

In Section \ref{s:weyl}, we introduce Weyl  almost periodic points.
Weyl almost periodicity of a point is substantially stronger than
Besicovitch almost periodicity. In fact, a point is Weyl  almost
periodic if and only if its orbit closure is uniquely ergodic and
has pure point spectrum with continuous eigenfunctions (Theorem
\ref{t:main-weyl}). In this case all points in the orbit closure are
Weyl almost periodic (Lemma \ref{l:orbit-closure}). To put these
results in perspective we note  that  \cite{DG,FGL} show that a
dynamical system is mean equicontinuous if and only if it is
uniquely ergodic with pure point spectrum and continuous
eigenvalues. Hence,  a point is Weyl almost periodic if and only if
its orbit closure is mean equicontinuous. So, our results can be
understood to provide  a natural pointwise counterpart to the
results of \cite{FGL}.

In Section \ref{s:Weakly-and-Bohr}, we investigate two special
classes of Weyl almost periodic points viz Bohr and weakly almost
periodic points respectively. This allows us to characterize Bohr
and weakly almost periodic dynamical systems. In particular, we
reprove a main result of \cite{LS2}.    An application of our
results to aperiodic order is given in Section \ref{s:application}.
This includes an alternative proof for the main results on measure
dynamical systems contained  in \cite{LeSpStr}.

Our article give a comprehensive treatment of most relevant concepts
of almost periodicity in the context of pure point spectrum.  Some
parts of our considerations allow for simple abstractions.  As this
may be of value for further investigations we include  a brief
discussion of some basic results in Section \ref{s:abstract}.

It seems that Besicovitch almost periodicity is not well known (for
groups other than $\RR$) and there is also some ambiguity in the way
it is defined. For this reason we include some  appendices
discussing the almost periodicity properties needed in this article
as well as basic properties of the class of continuous functions
with these almost periodicity properties.

\bigskip

{\small \textbf{Acknowledgments.} Parts of this work were done while
D.L. was visiting the department of mathematics in Geneva in 2019.
He would like to thank the department for hospitality. He would also
like to acknowledge most inspiring discussions with Gerhard Keller
and Michael Baake at the workshop on \textit{Model sets and
aperiodic order} in Durham in 2018. Part of the work was done when
T.S. and N.S. visited Friedrich-Schiller-Universit\"at Jena, and
they would like to thank the department of mathematics for
hospitality. The work was supported by DFG with grant 415818660(TS)
and by NSERC with grants 03762-2014 and 2020-00038(NS), and the authors are grateful
for the support.}

\section{Background on dynamical systems, pure point spectrum and the upper mean}
\label{s:mean}
In this section we review the necessary concepts from dynamical
systems and introduce the upper mean $\overline{M}$, which is
crucial for our subsequent considerations.

\bigskip

Throughout the paper, we denote the set of continuous complex valued functions
on the topological space $Y$ by $C(Y)$.

We consider  a compact metric  space $X$ equipped with a continuous
action
$$
\alpha: G\times X\longrightarrow X , \qquad (t,x)\mapsto
\alpha_t (x) ,
$$
of a locally compact, $\sigma$-compact,  abelian
group $G$. We then call $(X,G)$ a \textit{dynamical system (over the
space $X$)}. Often we will also be given a probability measure
$m$ on $X$, which is invariant under the action of $G$. We then call
$(X,G,m)$ a \textit{dynamical system} as well. We will write $tx$
instead of $\alpha_t (x)$ for $t\in G$ and $x\in X$. The composition
on $G$ itself is written additively and the neutral element of $G$
is denoted as $e$. We fix a Haar measure  on $G$. The Haar measure
of a measurable subset $A\subset  G$ is denoted by $|A|$ and the
integral of an integrable function $f$ on $G$ by $\int f(t)\, \dd t$.

Whenever a dynamical system $(X,G)$ is given  we will furthermore
make use of
\begin{itemize}
\item a metric $d$ on $X$ which generates the topology,
\item  a F\o lner sequence $(B_n)$ in $G$, i.e. each $B_n$ is
an open relatively compact subset of $G$  and
\[
\frac{|(B_n \setminus (t+B_n)) \cap ((t+B_n)\setminus B_n)|}{|B_n|}
\to 0,\qquad n\to \infty,
\]
for all $t\in G$.
\end{itemize}
 Note that a F\O lner sequence exists in a locally compact abelian
group $G$ if an only if $G$ is $\sigma$-compact \cite[Prop.~B6]{SS}.
For this reason, we will always assume that $G$ is $\sigma$-compact.

The \textit{orbit} of $x$ is given by $Gx:=\{tx : t\in G\}$ and the
\textit{orbit closure} $\overline{Gx}$ is the closure of the orbit.
If the orbit closure of  $x\in X$ agrees with $X$, the element $x$
is called \textit{transitive}. If every $x\in X$ is transitive the
dynamical system is called \textit{minimal}. The dynamical system
$(X,G)$ is called \textit{uniquely ergodic} if there exists only one
invariant probability measure on $X$. We then denote this measure by
$m$ and call  $(X,G,m)$ uniquely ergodic as well.

The dynamical system $(X,G,m)$ is \textit{ergodic} if any invariant
measurable subset $A$ of $G$ satisfies $m(A)=1$ or $m(A) =0$. If
$(X,G,m)$ is ergodic any F\o lner sequence has a subsequence $(B_n)$
for which \textit{Birkhoff ergodic theorem holds} i.e.
\[
\lim_{n\to \infty} \frac{1}{|B_n|} \int_{B_n} f(tx)\, \dd t = \int_X f \,
\dd m
\]
is valid for almost every $x\in X$ whenever $f:
X\longrightarrow \CC$  integrable.

Whenever a dynamical system $(X,G,m)$ and a F\o lner sequence
$(B_n)$ is given,  a point $y\in X$ is called $m$-\textit{generic}
with respect to the F\o lner sequence if
\[
\lim_{n\to \infty} \frac{1}{|B_n|} \int_{B_n} f(ty)\, \dd t =
\int_X  f(x) \, \dd m(x)
\]
holds for any continuous $f :
X\longrightarrow \CC$. If the measure $m$ and the sequence $(B_n)$
are clear from the context we just speak about generic points.
Generic points play a key role in our subsequent considerations as
they determine the measure and - in this sense - the whole dynamical
system. As is well-known (and not hard to see) the set of generic
points is measurable and invariant under the group action. Moreover,
the  set of generic points has full measure if $(X,G,m)$ is ergodic
and Birkhoff's ergodic theorem holds along the underlying F\o lner
sequence. While we do not need it here, it  is instructive for our
subsequent considerations to note that a converse of sorts holds as
well: If $m$ is an invariant probability measure such that the set
of $m$-generic points with respect to some F\o lner sequence has
full measure, then $(X,G,m)$ is ergodic. So, ergodicity is a
necessary and sufficient condition for having an ample supply of
generic points at ones disposal. This is ultimately the reason that
most (but not all)  of our theorems below deal with  ergodic
systems.

A dynamical system $(X,G,m)$ is said to have \textit{pure point
spectrum} if there exists an orthonormal basis of $L^2 (X,m)$
consisting of eigenfunctions. Here, an $f\in L^2 (X,m)$ with $f\neq
0$ is called an \textit{eigenfunction} if for any $t\in G$ there
exist a $\xi(t)\in \CC$ with
$$f(t\cdot) = \xi(t) f$$
in the sense of $L^2 (X,m)$ functions. In this case, each $\xi(t)$
belongs to the group $$\TT:= \{z\in \CC : |z| = 1\}$$ and the map
$$\xi: G\longrightarrow \TT ,\qquad t\mapsto \xi(t),$$
can easily be seen to be a continuous group homomorphism. It is
called  \textit{eigenvalue}. Clearly, eigenfunctions to different
eigenvalues are orthogonal. As $X$ is compact and metrizable, $L^2
(X,m)$ is separable. Hence,  the set of eigenvalues is (at most)
countable. For $\xi \in \widehat{G}$ we denote by $P_\xi$ the
projection onto the eigenspace of $\xi$ if $\xi$ is an eigenvalue
and set $P_\xi =0$ if $\xi$ is not an eigenvalue. The set of
eigenvalues of $(X,G,m)$  will be denoted by $\Eig(X,G,m)$. As is
well known the set of eigenvalues is a group if $(X,G,m)$ is
ergodic.

For our further discussion we will also  rely on some concepts
defined purely with respect to $G$ (i.e. they do not need the
dynamical system). Let $(B_n)$ be a F\o lner sequence, and let
$B(G)$ denote the set of bounded measurable real-valued functions on
$G$. Then, we define the associated \textit{upper mean}  via
\[
\overline{M}_{(B_n)}  : B(G)
\longrightarrow [0,\infty),\qquad \overline{M}_{(B_n)} (h) :=\limsup_{n\to
\infty} \frac{1}{|B_n|} \int_{B_n} h(s)\, \dd s.
\]
If the F\o lner
sequence is clear from the context we drop the subscript $(B_n)$.
Clearly,  $\overline{M}$ gives rise to  seminorm on the space of
bounded measurable functions on $G$ via $f\mapsto\overline{M}(|f|)$.

A subset $A$ of $G$ is called \textit{relatively dense} if there
exists a compact set $K\subset G$ with
$$ G =\bigcup_{a\in A} (a + K).$$

A continuous bounded $f : G\longrightarrow \CC$ is called
\textit{Bohr almost periodic} if for any $\varepsilon >0$ the set of
$t\in G$ with
$$\| f - f(\cdot -t)\|_\infty <\varepsilon $$
is relatively dense. Here, the supremum norm $\|\cdot\|_\infty$ for
bounded complex valued functions on $G$ is defined via $\|f\|_\infty
=\sup_{s\in G} |f(s)|$.

\section{Mean almost periodic points and pure point spectrum}\label{s:main}
In this section, we introduce and study mean almost periodic points.
The main result of this section then  provides a characterization of
pure point spectrum via mean almost periodicity of points.

\bigskip

Whenever $(X,G)$ is a dynamical system with metric $d$ and $(B_n)$
is a F\o lner sequence, we define
$$D = D_d^{(B_n)} : X\times
X\longrightarrow [0,\infty),\qquad \: D(x,y):=\overline{M}( s \mapsto
d(sx, sy)).$$ Clearly, $D$ is a pseudometric. Moreover, the F\o lner
condition on $(B_n)$ easily gives that $D$ is invariant i.e.
satisfies $D(tx, ty) = D(x,y)$ for all $x,y\in X$ and $t\in G$.
Furthermore, for each $x\in X$ the function $G\to [0,\infty)$,
$t\mapsto D(x,tx)$, is uniformly continuous. Indeed, continuity at
$t=0$ is easily seen from the definition. Moreover, due  to the
invariance and the pseudometric properties we infer
\begin{equation}
|D(x,tx) - D(x,sx)|\leq D(tx, sx) = D(x, (s-t)x).  \tag{$\clubsuit$}
\end{equation}
When combined with continuity at $t =0$, this gives uniform
continuity. We will refer to $D$ as the \textit{averaged metric on
$X$ associated to $d$ and $(B_n)$}.

\begin{definition}[Mean almost periodic points] Let $(X,G)$ be a dynamical
system and $d$ a
metric on $X$ generating the topology and $(B_n)$ a F\o lner
sequence and $D$ the associated averaged metric. Then, a point $x\in
X$ is called  mean almost periodic with respect to $d$ and $(B_n)$
if for every $\varepsilon >0$ the set
$$\{t\in G : D(x,tx)< \varepsilon\}$$
is relatively dense.
\end{definition}

\begin{remark}  Almost periodicity properties with respect to
$\overline{M}$ are often connected with the  name of
\textit{Besicovitch}.  We use this for a strengthened version to be
introduced below. Here we stick to the term `mean' as this seems to
be the common term within the study of equicontinuity properties in
recent years (see e.g. \cite{DG,FGL,GM} as well as discussion in
Appendix \ref{s:a:beap}).
\end{remark}

By definition, mean almost periodicity depends on the chosen F\o lner
sequence. In our subsequent discussion of mean almost periodic
points, however,  we will often refrain from  explicitly referring
to the F\o lner sequence $(B_n)$ if it is clear from the context
which sequence is involved.

\begin{lemma}\label{l:map-via-bohr} Let $(X,G)$ be a dynamical
system and $d$ a metric on $X$ generating the topology. An $x\in X$
is mean almost periodic if and only if the function $G\ni t\mapsto
D(x,tx)\in \RR$ is Bohr almost periodic.
\end{lemma}
\begin{proof} Define $f$ on $G$ via $f(t) := D(x,tx)$. We have already noted that the function $f$ is uniformly continuous.
Clearly, Bohr almost periodicity of $f$ implies that $x$ is mean
almost periodic (as $f(0) =0$). Conversely, $(\clubsuit)$ gives
$|f(t +s) - f(s)| \leq f(t)$ for all $s,t\in G$ and  mean almost
periodicity of $x$ implies Bohr almost periodicity of $f$.
\end{proof}

Our next aim is to discuss independence of this definition from the
metric and to provide an alternative way of defining mean almost
periodicity via density of super level sets. The proofs of the
corresponding statements are not difficult and rather close to each
other. They rely on some simple facts stated in the next
proposition. We denote the characteristic function of a set
$A\subset G$ by $1_A$ (i.e. $1_A (x) = 1$ for $x\in A$ and $1_A (x)
=0$ for $x\notin A$).

\begin{prop}\label{p:simple} Let   $f :
G\longrightarrow [0,1]$  be given.  Then, for any $\delta>0$ the
following estimates hold for the set
$$A(f,\delta):=\{s\in G : f(s)
\geq \delta\}:$$

(a)  $\overline{M}( 1_{A(f,\delta)}) \leq \frac{1}{\delta} \;
\overline{M}(f)$.

(b) $\overline{M}(f) \leq M(1_{A(f,\delta)}) + \delta.$
\end{prop}
\begin{proof} (a) We clearly have $ 1_{\{s\in G :
f(s) \geq \delta\}} \leq \frac{1}{\delta} f$. This  easily gives
(a).

\smallskip

(b) We compute
$$ \overline{M}(f) \leq  \overline{M}( f \cdot 1_{\{s\in G : f(s) \geq \delta\}}) +\overline{M}(
f\cdot  1_{\{s\in G : f(s) < \delta\}}) \leq \overline{M}( 1_{\{s\in
G : f(s) \geq \delta\}}) + \delta.$$ This finishes the proof.
\end{proof}

\begin{lemma}[Independence of the metric] \label{l:independence} Let $(X,G)$ be a dynamical system
with metrizable $X$ and $(B_n)$ a F\o lner sequence on $G$.
Then,  mean almost periodicity of an $x\in X$ does not depend on the
chosen metric (provided it generates the topology).
\end{lemma}
\begin{proof}
Let $e,d$ be two metrics on $X$, which generate the topology. For
$x\in X$ and $t\in G$ we define the functions $d_{t,x}$ and
$e_{t,x}$ on $G$ via $d_{t,x}(s) = d(sx,tsx)$ and $e_{t,x}(s) =
e(sx,tsx)$, respectively.  We show that for any $\varepsilon >0$
there exists a $\delta >0$ such that for any $t\in G$ and $x\in X$
we have
$$\overline{M} (d_{t,x}) <\varepsilon$$
whenever $\overline{M}(e_{t,x}) < \delta$ holds. The statement with
the roles of $e$ and $d$ reversed can be shown analogously and taken
together these two statements prove the lemma.

\smallskip

Without loss of generality we assume $d,e\leq 1$.

\smallskip

Let $\varepsilon >0$ be given.  Choose $\delta'
>0$ with $d(z,y) < \frac{\varepsilon}{2}$ whenever $e(z,y) <
\delta'$. Set
$$\delta:= \delta' \cdot \frac{\varepsilon}{2}.$$
If $\overline{M}(e_{t,x}) < \delta$ then (a) of Proposition
\ref{p:simple} gives
$$ \overline{M}(1_{\{ s : e_{t,x}(s) \geq \delta'\}}) \leq
\frac{1}{\delta'} \overline{M} (e_{t,x}) < \frac{\delta}{\delta'} =
\frac{\varepsilon}{2}.$$ Furthermore we note that, by definition of
$\delta'$, we have
$$\overline{M}(d_{t,x}  1_{\{s  : e_{t,x}(s) < \delta'\}})\leq
\frac{\varepsilon}{2}.$$ Given this we can now estimate

$$\overline{M} (d_{t,x}) \leq \overline{M} (d_{t,x}  1_{\{ s : e_{t,x} (s) \geq
\delta'\}})) + \overline{M} (d_{t,x} (s) 1_{\{ s : e_{t,x}
<\delta'\}}) <  \frac{\varepsilon}{2} + \frac{\varepsilon}{2} =
\varepsilon.$$
This is the desired statement.
\end{proof}

By the previous lemma  mean almost periodicity of a point is
independent of the underlying metric. Hence, we can (and will) from
now on refrain from specifying a metric when talking about mean
almost periodicity.

\smallskip

We define the \textit{upper density} of a subset $A\subset G$ via
$$\mbox{Dens}(A):=\overline{M} (1_A).$$

\begin{lemma}[Mean almost periodicity via density of superlevel
sets]\label{l:besi-distance} Let $(X,G)$ be a dynamical system with
metrizable $X$ and $(B_n)$ a F\o lner sequence on $G$. Then, the
following assertions for $x\in X$ are equivalent:
\begin{itemize}

\item[(i)] The point $x$ is mean almost periodic.

\item[(ii)] For any $\delta>0$ the set of $t\in G$
with
$$
\operatorname{Dens}(\{s\in G : d(sx,tsx)\geq \delta\})< \varepsilon
$$
is relatively dense for any $\varepsilon >0$.
\end{itemize}

\end{lemma}

\begin{proof} We use the notation of the proof of the preceding
lemma.

\smallskip

(i)$\Longrightarrow$(ii): Let $\delta>0$ and $\varepsilon >0$ be
arbitrary. By (i) the set of  $t\in G$ with $\overline{M}(d_{t,x})<
\delta \varepsilon$ is relatively dense. Now, for any such $t\in G$
we obtain from (a) of Proposition \ref{p:simple}
$$ \overline{M}(1_{\{ s : d_{t,x}(s) \geq \delta\}}) < \varepsilon.$$

\smallskip

(ii)$\Longrightarrow$(i): Let $\varepsilon >0$ be given. Assume
without loss of generality that $d\leq 1$. Set $\delta =
\frac{\varepsilon}{2}$. By (ii) the set of  $t\in G$ with $
\mbox{Dens}(\{s\in G : d(sx,tsx)\geq \delta\})<
\frac{\varepsilon}{2}$ is relatively dense. Choose such a $t\in G$.
Then, (b) of Proposition \ref{p:simple} gives $$\overline{M}
(d_{t,x}) \leq \overline{M} ( 1_{\{ s : d_{t,x} (s) \geq \delta\}})
+ \delta =   \mbox{Dens}(\{s\in G : d(sx,tsx)\geq \delta\}) + \delta
<  \frac{\varepsilon}{2} + \frac{\varepsilon}{2} = \varepsilon.$$
This finishes the proof.
\end{proof}

We finish this section with the discussion of a further
characterization of mean almost periodicity via suitable functions.
To state this characterization (and similar characterizations in
subsequent sections) it is useful to define for $x\in X$ and $f\in
C(X)$ the function
$$f_x : G\longrightarrow \CC, \qquad  f_x (t) = f(tx).$$
Moreover, we set
$$\mathcal{A}_x :=\{ f_x : f\in C(X)\}.$$
Clearly, $\mathcal{A}_x$ is an algebra.

\begin{prop}[Completeness of $\mathcal{A}_x$] The algebra $\mathcal{A}_x$
 is complete with respect to $\|\cdot\|_\infty$.
 \end{prop}
 \begin{proof}
Consider a sequence $(f^{(n)})$ in $C(X)$ such that $(f^{(n)}_x)$ is
a Cauchy sequence with respect to $\|\cdot\|_\infty$.  Then, a
direct $\varepsilon/3$ argument shows that the restrictions of
$f^{(n)}$ to the orbit closure of $x$ converge uniformly  to a
continuous function on the orbit closure. Now, the desired statement
follows from Tietze's extension theorem.
\end{proof}

A bounded measurable function $f : G\longrightarrow \CC$ is
\textit{mean almost periodic} with respect to $(B_n)$  if for every
$\varepsilon
>0$ the set
$$\{t\in G: \overline{M}(|f(\cdot) - f(\cdot - t)|)<\varepsilon \}$$
is relatively dense. The set of { uniformly continuous} mean almost
periodic functions is an algebra and closed under complex
conjugation and uniform convergence (see appendix \ref{s:a:map}).

As a consequence of the previous considerations we can now
characterize mean almost periodicity via functions.

\begin{lemma}[Mean almost periodicity via functions]\label{l:char-ap-functions}
Let $(X,G)$ be a dynamical system  and $(B_n)$ a F\o lner sequence
on $G$. For $x\in X$ the following assertions are equivalent:
\begin{itemize}
\item[(i)] The point $x$ is mean almost periodic.
\item[(ii)] Every element from  $\mathcal{A}_x$ is mean almost
periodic.
\item[(iii)] The set $\{f\in C(X) : f_x \mbox{  is mean almost periodic}\}$
separates the points of $\overline{Gx}$.
\item[(iv)]
 For any $s\in G$ the function $d^{(s)}_x$ is mean almost periodic, where
 $d^{(s)} \in C(X)$ is defined via $d^{(s)} (y) : =d(sx,y)$.
\end{itemize}
\end{lemma}
\begin{remark} As the proof shows condition (iii) could equivalently
be formulated with $\overline{Gx}$ replaced by the whole space $X$.
\end{remark}

\begin{proof} (iv)$\Longrightarrow$(iii): This follows as the $d^{(s)}$, $s\in G$,
clearly separate the points of $\overline{Gx}$.

\smallskip

(iii)$\Longrightarrow$(ii):  Invoking the corresponding properties
of mean almost periodic functions,  we can easily see  that $\{f \in
C(X) : f_x \mbox{ is mean almost periodic}\}$ is an algebra, which
is closed under complex conjugation and uniform convergence. This
algebra clearly contains the constant functions. Moreover, by
assumption (iii) it separates the points of $\overline{Gx}$.
Furthermore this algebra contains every function $f\in  C(X)$, which
vanishes on $\overline{Gx}$ (as $f_x = 0$ for any such functions).
Thus, this algebra even separates the points of $X$.
 Now, (ii) follows from Stone--Weierstra\ss{}' theorem.

\smallskip

(ii)$\Longrightarrow$(i): Choose a countable set $\mathcal{C}
\subset C(X)$ such that any $f\in \mathcal{C}$ satisfies
$\|f\|_\infty \leq 1$ and such that the elements of $\mathcal{C}$
separate the points of $X$. Let $c_f>0$, $f\in \mathcal{C}$, with
$\sum_{f\in \mathcal{C}} c_f <\infty$ be given. Then,
$$e (z,y) := \sum_{f\in \mathcal{C}} c_f\, |f(x) - f(y)|$$
is a metric on $X$, which generates the topology. Moreover, by
assumption (ii) the function $f_x$ is mean almost periodic for any
$f\in C(X)$ and, hence, any $f\in \mathcal{C}$.  This easily gives
that the set of $t\in G$ with $\overline{M}(s\mapsto e(sx,(t+s)x))
<\varepsilon$ is relatively dense (compare Proposition \ref{p:sum}).
Thus, $x$ is mean almost periodic with respect to the metric $e$.
As mean almost periodicity does not depend on the metric, we
conclude (i).

\smallskip

(i)$\Longrightarrow$(iv): Let $z\in X$ be arbitrary and define
$d^{(z)} \in C(X)$ by $d^{(z)}  (y):= d(z,y)$.  The triangle
inequality for $d$ gives
$$\overline{M}( |d^{(z_x)} (\cdot - t) - d^{(z)}_x|)= \overline{M}( s\mapsto |d(z,(s-t)x) - d(z,sx)|)
\leq \overline{M}(s\mapsto d(sx, (s-t)x)) = D(x,-tx)$$ for any $t\in
G$ and $z\in X$. Now, mean almost periodicity of $d^z_x (\cdot)$
follows (for any $z\in X$) from (i).

\end{proof}

We now come to the main result of this section, which provides a
characterization of pure point spectrum via mean almost periodic
points.

\begin{theorem}\label{t:main}
Let  $(X,G,m)$ be a dynamical system  and $(B_n)$ a F\o lner
sequence on $G$. Assume that there exists a generic point for
$(X,G,m)$. Then, the following assertions are equivalent:

\begin{itemize}
\item[(i)] The dynamical system $(X,G,m)$ has pure point
spectrum.

\item[(ii)] Every generic point of $X$ is mean almost periodic.

\item[(iii)] One  generic point of $X$ is mean almost periodic.

\end{itemize}
If  $(X,G,m)$ is ergodic and the Birkhoff theorem holds along
$(B_n)$, these statements are also equivalent to the following
statement:

\begin{itemize}

\item[(iv)] Almost every $x\in X$ is  mean almost periodic.

\end{itemize}

\end{theorem}
\begin{remark}   A related result for
subshifts over a finite alphabet   can be found in Lemma 5 of
\cite{Ver}. There, pure point spectrum (i) is characterized  via a
variant of (iv) given by  a mean almost periodicity condition on
points defined via a metric (close in spirit to what is discussed
above in Lemma \ref{l:besi-distance}).
\end{remark}
\begin{proof} In the ergodic case
almost every point is generic. Hence, (ii)$\Longrightarrow$(iv) and
(iv)$\Longrightarrow$(iii) follow.  So we now turn to showing
equivalence between (i), (ii) and (iii) in the general case. We
clearly have (ii)$\Longrightarrow$(iii). To show
(i)$\Longrightarrow$(ii) and (iii)$\Longrightarrow$(i) we define
\[
\underline{d} : G\longrightarrow [0,\infty),\qquad \underline{d} (t) =
\int_X d(x,tx)\, \dd m(x).
\]
The main result of  \cite{Len} says that
(i) is equivalent to $\underline{d}$ being Bohr almost periodic.
Thus, it remains  to show that
\begin{itemize}
\item Bohr almost periodicity of $\underline{d}$ implies
(ii);

\item  (iii) implies Bohr almost periodicity of
$\underline{d}$.

\end{itemize}
Now, for $t\in G$ we can consider $f_t : X\longrightarrow
[0,\infty),\ f_t (x) = d(x,tx)$. Then, $f_t$ is clearly continuous.
Thus, whenever $y\in X$ is generic, we find
\[
\underline{d}(t) = \int_X d(x,tx)\, \dd m(x) =\int_X f_t (x)\, \dd m (x) =
\overline{M}  (s \mapsto f_t(sy))=D(y,ty)
\]
for every $t\in G$.
Moreover, the triangle inequality gives that $\underline{d}$ is Bohr
almost periodic if and only if the set
$$\{t\in G : \underline{d}(t) <\varepsilon\}$$
is relatively dense in $G$ for all $\varepsilon >0$. Putting this
together we easily obtain that $\underline{d}$ is Bohr almost
periodic if and only if one (every) generic $y\in X$ is mean almost
periodic.
\end{proof}

We emphasize that the first part of the preceding theorem does not
need an ergodicity assumption and illustrate this by the following
example.

\medskip

\textbf{Example - pure point spectrum in non ergodic case.} Consider
$\{0,1\}$ with discrete topology and  $X = \{0,1\}^\ZZ$ with product
topology. Equip $X$ with the shift action of $\ZZ$ given by
$\alpha_n(x) = x(\cdot - n)$ for $n\in\ZZ$. Let $\underline{1}$ and
$\underline{0}$ be the elements of $X$ which are constant equal to
$1$ and $0$ respectively. Then, clearly each of these elements is
invariant under the shift action and so are then the sets
$\{\underline{0}\}$ and $\{\underline{1}\}$. Thus,
$$m:=\frac{1}{2}(\delta_{\underline{0}} +
\delta_{\underline{1}})$$ is an invariant probability measure (where
$\delta_p$ denotes the unit point mass at $p$). Obviously, $m$ is
not ergodic. The space $L^2(X,m)$ is two-dimensional and $\sqrt{2}
\cdot 1_{\{\underline{0}\}}, \sqrt{2} \cdot 1_{\{\underline{1}\}}$
is an orthogonal basis consisting of eigenfunctions (to the
eigenvalue $1$). In particular, $(X,\ZZ,m)$ has pure point spectrum.
Now, consider the point $x\in X$ with $x(-k)$ arbitrary for $k\geq
0$ and $x(k) = 1 $ if $k\in \{2^n,\ldots,  2^n +2^{n-1}-1\}$ for
some $n\in \NN$ and $x(k) =0$ else. Then, it is not hard to see that
$x$ is generic for $m$ with respect to the F\o lner sequence $B_n
=\{1,\ldots, 2^n\}$. So, the theorem gives that $x$ is mean almost
periodic. In this example neither $\underline{0}$ nor
$\underline{1}$ are generic. Hence, $m$ does not give mass to
generic points and  the set of generic points has measure zero. Note
that the construction of $x$  could easily be modified to yield a
transitive generic point (by including suitable finite words of
slowly increasing length  between the blocks of $1's$ and $0's$ in
$x$).

\medskip

Combining the previous result,  Theorem \ref{t:main}, with the
characterization of mean almost periodicity via functions in Lemma
\ref{l:char-ap-functions} we obtain the  following.

 \begin{coro}\label{c:char-point-separating}
Let $(X,G,m)$ be an ergodic dynamical system with metrizable $X$ and
assume that the Birkhoff ergodic theorem holds along $(B_n)$.
 Then, the following assertions are equivalent.
 \begin{itemize}
\item[(i)] The dynamical system $(X,G,m)$ has  pure point spectrum.
\item[(ii)] For almost every $x\in X$ the set $\{f \in C(X) : \mbox{$f_x$
 is mean almost periodic}\}$  separates the points of $X$.
 \end{itemize}
\end{coro}

\begin{remark}
A variant of this   statement (with (ii) replaced by  the stronger
condition that $\mathcal{A}_x$ consists only of mean almost periodic
functions)  is shown in \cite{Len1} based on an earlier version of
\cite{LeSpStr}. Our proof is different.  Note also that
 for ergodic systems over $G = \ZZ$, it is known that pure point
spectrum is equivalent to $\ZZ\ni n\mapsto f(nx)$ belonging to the
Besicovitch class for almost every $x\in X$ whenever $f$ is a
bounded measurable function on $X$, see Theorem 3.22 in \cite{BeLo}.
The condition of Besicovitch class is stronger than mean almost
periodicity (see also next section).
\end{remark}

If the system $(X,G,m)$ is uniquely ergodic,  then every $x\in X$ is
generic irrespective of the underlying F\o lner sequence (Oxtoby's
theorem). Thus, from the previous theorem we obtain immediately the
following corollary.

\begin{coro} Let  $(X,G,m)$ be a dynamical system  and $(B_n)$ a F\o lner
sequence on $G$. Assume that $(X,G,m)$ is uniquely  ergodic. Then,
the following assertions are equivalent:

\begin{itemize}
\item[(i)] The dynamical system $(X,G,m)$ has pure point
spectrum.
\item[(ii)] Every $x\in X$ is  mean almost periodic.
\item[(iii)] One $x\in X$ is mean almost periodic.
\end{itemize}

\end{coro}
\begin{remark} The concept of mean almost periodicity depends on the chosen F\o lner sequence.
To see this, consider $X:=\{0,1\}^\ZZ$ with product topology and the
shift action of $\ZZ$ and the Bernoulli measure $m$ (product measure
of the measures giving equal weights $1/2$ to $\{0\}$ and $\{1\}$).
This system is ergodic and  $m$ almost every $x\in X$ contains
arbitrary long stretches of $0's$. For each of those $x$ we can then
choose a F\o lner sequence $(B_n)$ with $\overline{M}(s\mapsto
d(sx,(t+s)x) = 0$ for all $t\in \ZZ$ (by each  $B_n$ being chosen
`within' a long stretch of $0's$ with distance to the boundary of
these stretches increasing in $n$). Hence, each of these $x$ is mean
almost periodic. On the other hand, as the system does not have pure
point spectrum, we obtain from Theorem \ref{s:main} that  not every
of these $x$ will be almost periodic with respect to the standard
F\o lner sequence $B_n = \{0,\ldots,n\}$ along which Birkhoff's
ergodic theorem holds.
\end{remark}

\section{Besicovitch almost periodic points and eigenfunctions}
\label{s:strengthening}
In this section, we discuss a strengthened
version of mean almost periodicity viz Besicovitch almost
periodicity. We  show that pure point spectrum can also be
characterized via this strengthened version. In fact, our results
can be understood as saying that in a dynamical system with pure
point spectrum both eigenfunctions and eigenvalues can be read of
from any of its (generic) Besicovitch almost periodic points.

\bigskip

We consider a { $\sigma$-compact}, locally compact abelian group
together with a F\o lner sequence $(B_n)$. As usual the set of all
continuous group homomorphisms $\xi : G\longrightarrow \TT$ is
denoted as $\widehat{G}$ and called the \textit{dual group} of $G$.
We say that a  bounded function $f : G\longrightarrow \CC$ is
\textit{Besicovitch almost periodic} if for any $\varepsilon
>0$ there exist $k\in \NN$, $\xi_1,\ldots, \xi_k\in
\widehat{G}$ and $c_1,\ldots ,c_k\in\CC$ with
$$\overline{M} (|f - \sum_{j=1}^k c_j \xi_j|)< \varepsilon.$$
A discussion of basic properties of { uniformly continuous}
Besicovitch almost periodic functions is given in Appendix
\ref{s:a:beap}. This shows in particular that any { uniformly
continuous} Besicovitch almost periodic function is also mean almost
periodic and admits an average (see below as well). The discussion
also shows that the set of these functions forms an algebra and is
closed under uniform convergence.

\begin{definition}[Besicovitch almost periodic points]
Let $(X,G)$ be a metrizable dynamical system  and $(B_n)$ a F\o lner
sequence. Then, a point $x\in X$ is called Besicovitch almost
periodic with respect to $(B_n)$ if $\mathcal{A}_x$ consists only of
Besicovitch almost periodic functions.
\end{definition}

As in the definition of  mean almost periodicity also Besicovitch
almost periodicity  depends on the chosen F\o lner sequence. In our
subsequent discussion, however,  we will often refrain from
explicitly referring to the F\o lner sequence $(B_n)$ if it is clear
from the context which sequence is involved.

\begin{remark} (a)  To set this definition in perspective we refer to
Lemma \ref{l:char-ap-functions}. This lemma shows that a point is
mean almost periodic if and only if $\mathcal{A}_x$ consists of mean
almost periodic functions only.

 (b)  Note also that the statements of Lemma
\ref{l:char-ap-functions} remain true (with essentially the same
proof) after 'mean almost periodic' is replaced with 'Besicovitch
almost periodic'.

(c) As Besicovitch almost periodic functions are mean almost
periodic,  any Besicovitch almost periodic point is mean almost
periodic. The converse is not true. To see this consider
$X:=\{0,1\}^\ZZ$ with product topology and the shift action of
$\ZZ$. Set $B_n :=\{0,\ldots, n\}$ for $n$ even and $B_n =
\{-n,\ldots, -1\}$ for $n$ odd. Consider now $y\in X$ with $y(k) = 1
$ for $k\geq 0$ and $y(k) = 0$ otherwise. Then, it is not hard to
see that $D(y,ny) = 0$ for all $n\in\ZZ$. Hence, $y$ is mean almost
periodic. On the other hand, consider  $f: X\longrightarrow \{0,1\}$
with $f(x) = x(0)$. Clearly $f$ is continuous. Moreover,
$$a_n :=\frac{1}{|B_n|} \sum_{k\in B_n} f(k y)$$
does not converge (as $a_{2n} =1$ and $a_{2n+1} =0$ for all
$n\in\NN$). By the discussion in the appendix (see subsequent
proposition as well),  this shows that $f_y$ is not Besicovitch
almost periodic. Hence, $y$ is not Besicovitch almost periodic.
\end{remark}

Here comes a  characterization of  Besicovitch almost periodic
points via existence of means. To state it, we first introduce some
notation. For a bounded measurable function $h: G\longrightarrow \CC$, we define
the \textit{mean} or \textit{average} of $h$ with respect to $(B_n)$
by
\[
\Average(h):= \lim_{n\to \infty} \frac{1}{|B_n|} \int_{B_n} h(t)\, \dd
t
\]
whenever the limit exists.

\begin{prop}[Averaging along orbits]\label{p:average}
Let $(X,G)$ be  a dynamical system and $(B_n)$
be a F\o lner sequence. Then, the following assertions are
equivalent for $x\in X$:
\begin{itemize}
\item[(i)] The point $x$ is Besicovitch almost periodic.

\item[(ii)]
For any $f\in C(X)$ there exists a countable set $F_f\subset
\widehat{G}$  such that  the limits
\[
\Average( |f_x|^2) = \lim_{n\to \infty} \frac{1}{|B_n|} \int_{B_n}
|f(tx)|^2 \, \dd t \mbox{ and }  \Average(f_x \overline{\xi})
=\lim_{n\to \infty} \frac{1}{|B_n|} \int_{B_n} f(tx)\,
\overline{\xi(t)} \, \dd t
\]
exist for all $\xi \in F_f$  and
\[
\Average(|f_x|^2)  =\sum_{\xi \in F_f} |\Average(f_x \overline{\xi})|^2
\]
holds.
\end{itemize}
Moreover, in case that (i) and (ii) hold, $A(f\overline{\xi})$
exists and equals $0$ for   $f\in C(X)$ and $\xi \in
\widehat{G}\setminus F_f$.
\end{prop}
\begin{proof} This is a  direct consequence of the definition of Besicovitch almost periodicity
of a point and Proposition \ref{p:fourier-expansion} in Appendix
\ref{s:a:beap}.
\end{proof}

\begin{definition}[Frequency] Let $(X,G)$ be a dynamical system and $x\in X$ a
Besicovitch almost periodic point. Then, every  $\xi \in\widehat{G}$
with $\Average( f_x\overline{\xi}) \neq 0$ for some $f\in C(X)$ is
called a frequency of $x$. The set of all frequencies of $x$ is
denoted by $\Freq(x)$.
\end{definition}

Here is the first main result of this section.  It shows that any
Besicovitch almost periodic point completely determines a dynamical
system with pure point spectrum.

\begin{theorem}\label{t:converse}
Let $(X,G)$ be a dynamical system, $(B_n)$ a F\o lner sequence and
$p\in X$  a Besicovitch almost periodic point. Then, there exists a
(unique) ergodic probability measure $m$ on $X$ such that $p$ is
generic with respect to $m$. The  dynamical system $(X,G,m)$ has
pure point spectrum and $\Eig(X,G,m) =\Freq(p)$ holds. To each
eigenvalue $\xi \in \Eig(X,G,m)$ there exists a (unique)
eigenfunction $e_\xi\in L^2 (X,m)$ with
\[
\int_X f\, \overline{e_\xi}\, \dd m = \Average( f_p \overline{\xi})
\]
for all $f\in C(X)$.
\end{theorem}
\begin{proof}
Obviously, the map
\[
\Phi : C(X)\longrightarrow \CC,\qquad f\mapsto \Average(f_p),
\]
is linear and positive (i.e. $\Average(f_p)\geq 0$ for $f\geq0$).
Hence, there exist a unique measure $m$ on $X$ with $\Phi (f) =
\int_X f\, \dd m$. Clearly, $m(X) = \int_X 1\, \dd m = \Average(1) =
1$. By the F\o lner property of $(B_n)$,   the mean $\Average$ is
invariant and so is then $\Phi$. This easily gives that $m$ is
invariant. So, $m$ is an invariant probability measure.

We now turn to the construction of the eigenfunctions. For $\xi\in
\widehat{G}$ consider the map
\[
\Phi_\xi : C(X)\longrightarrow \CC,\qquad f\mapsto
\Average(f_p\overline{\xi}).
\]
This map is obviously linear and
defined on a dense subspace of $L^2 (X,m)$. By Cauchy-Schwarz
inequality, $\Average(1) =1$ and the construction of $m$ we find
\[
|\Phi_\xi (f)|^2 =|\Average(f_p \overline{\xi})|^2 \leq
\Average(|f_p|^2) \Average(1) = \int_X |f|^2\, \dd m.
\]
Hence, $\Phi_\xi$ can be extended to a linear
continuous map, again denoted by $\Phi_\xi$, on the whole $L^2
(X,m)$.  By Riesz Lemma there exists then  an $e_\xi\in L^2 (X,m)$
with $\|e_\xi\|\leq 1$ and
\[
\Phi_\xi(f) = \int_X f\, \overline{e_\xi}\, \dd m
\]
for all $f\in C(X)$. Define
\[
E:=\{\xi\in \widehat{G} : e_\xi \neq 0\}.
\]
By construction $\xi \in\widehat{G}$ belongs to $E$ if and only if
there exists an $f\in C(X)$ with $\Average(f_p \overline{\xi}) \neq
0$. Hence, $E= \Freq(p)$ holds. In particular, we have $\Average(f_p
\overline{\varrho})=0$ for all $f\in C(X)$ and $\varrho \in
\widehat{G}\setminus E$. A short computation shows for $t\in G$
\begin{eqnarray*}
\xi(-t) \int_X f\,\overline{ e_\xi}\, \dd m &=& \xi(-t) \Average(
f_p
\overline{\xi})\\
&=& \Average(f_p \overline{\xi(t +\cdot )})\\
(\mbox{$\Average$ invariant})\;\: &=& \Average(f_p(-t\cdot) \overline{\xi})\\
(\mbox{construction of $m$})\;\: &=& \int_X f(-t\cdot)\, \overline{e_\xi}\, \dd m\\
(\mbox{$m$ invariant})\;\: &=& \int_X f\, \overline{
e_\xi(t\cdot)}\, \dd m
\end{eqnarray*}
for all $f\in C(X)$. As these $f$ are dense in $L^2(X,m)$ this gives
$$ e_\xi (t\cdot)  = \xi(t) e_\xi$$
for all $t\in G$. This shows that $e_\xi$ is an eigenfunction (to
$\xi$) for each $\xi\in E$. Clearly, eigenfunctions to different
eigenvalues are orthogonal.

Next, we show that the $e_\xi$, $\xi \in E$,  are normalized and
form a basis. This gives that $E=\Eig(X,G,m)$ and together with the
already shown  $E = \Freq(p)$,  this will  then also imply
$\Eig(X,G,m) = \Freq(p)$.

By Parseval inequality, the definition of $m$ and Proposition
\ref{p:average}, we  have the following:
\begin{eqnarray*}
\sum_{\xi\in E} \left| \int_X f\, \overline{e_\xi}\, \dd m\right|^2  &\leq & \int_X |f|^2\, \dd m\\
&=& \Average(|f_p|^2)\\
(\mbox{Proposition \ref{p:average}})\;\: &=& \sum_{\xi\in
\widehat{G}}
|\Average(f_p\overline{\xi})|^2\\
(\mbox{construction of $E$})\;\: &=& \sum_{\xi\in E}
|\Average(f_p\overline{\xi})|^2\\
(\mbox{construction of $e_\xi$})\;\: &=& \sum_{\xi\in E}
\left|\int_X f\,\overline{e_\xi}\, \dd m\right|^2.
\end{eqnarray*}
This shows
\[
\sum_{\xi\in E} \left| \int_X f\, \overline{e_\xi}\, \dd m\right|^2
= \int_X |f|^2\, \dd m
\]
for all $f\in C(X)$. This is only
possible if $\|e_\xi\| =1$ for all $\xi\in E$ and $(e_\xi)$ form an
orthonormal basis of $L^2 (X,m)$.

\smallskip

It remains to show ergodicity: For each eigenvalue $\xi \in E$ we
have  constructed an eigenfunction $e_\xi$  and we have shown that
these form a complete set (i.e. $e_\xi$, $ \xi \in E$, is an
orthonormal basis). Hence (as different of these eigenfunctions
belong to different eigenspaces) each eigenspace is one-dimensional.
In particular the eigenspace to the eigenvalue $1$ is one
dimensional and the system is ergodic.
\end{proof}

\begin{remark}  Note that the proof relies (and only relies) on the
characterizing properties of Besicovitch almost periodic points
given in Proposition \ref{p:average}.
\end{remark}

The previous theorem shows that any Besicovitch almost periodic
point is generic with respect to a (uniquely determined) measure. It
may well be that different Besicovitch almost periodic points are
generic with respect to different measures. Consider e.g. the full
shift $X=\{0,1\}^{\ZZ}$ with $T x (n) = x(n+1)$. Then, any periodic
element of $X$ is Besicovitch almost periodic. Clearly, elements
with different periods will not be generic with respect to the same
measure. This motivates the following definition.

\begin{definition}[Generic Besicovitch almost periodic points]
Let $(X,G)$ be a dynamical system and $(B_n)$ a F\o lner sequence on
$G$. Then, for any invariant probability measure $m$ on $(X,G)$ we
denote by $\Bap (X,G,m)$ the set of those Besicovitch almost
periodic points which are generic with respect to $m$.
\end{definition}

We have the following consequence of the preceding theorem.

\begin{prop}
Let $(X,G,m)$ be a dynamical system and $(B_n)$ be a F\o lner
sequence. Then, $\Bap (X,G,m)$ is a measurable invariant set.
\end{prop}
\begin{proof} Clearly,  $\Bap (X,G,m)$ is invariant as both the set of generic points and
the set of Besicovitch almost periodic points are invariant. If
$\Bap (X,G,m)$ is empty there is nothing left to show. So, consider
the case  $\Bap(X,G,m)\neq\varnothing$. By Theorem \ref{t:converse},
the dynamical system $(X,G,m)$ then has pure point spectrum and the
set of its eigenvalues  $\Eig(X,G,m)$ equals $\Freq(p)$ for any
$p\in \Bap(X,G,m)$. Set $E:= \Eig(X,G,m)$.

\smallskip

\textit{Claim.} Let $D$ be a dense subset of $C(X)$. Then, we have
$p\in \Bap(X,G,m)$ if and only if the following three points hold:
\begin{itemize}
\item $A(f_p)$ exists and equals $\int_X f\, \dd m$ for all $f\in D$.
\item $A(f_p \overline{\xi})$ exists for all $\xi \in E$ and $f\in D$.
\item $A(|f_p|^2) = \sum_{\xi \in E} |A(f_p\overline{\xi})|^2$ for
all $f\in D$.
\end{itemize}
\textit{Proof of claim.} Consider $p\in \Bap (X,G,m)$. Then, $p$ is
generic and the first point holds (even for all $f\in C(X)$). In
particular, $A(|f_p|^2)$ exists. Now, the second and third point
follow from  Proposition \ref{p:average} as $p$ is Besicovitch
almost periodic with set of frequencies given by $E$.

\smallskip

Consider now a $p\in X$ satisfying the three points above. By
density of $D$ in $C(X)$ we  then easily infer that $A(f_p) = \int_X
f\, \dd m$ holds for all $f\in C(X)$ and $A(f_p \overline{\xi})$
exists for all $f\in C(X)$ and  $\xi \in E$. In particular, we have
$A(|f_p|^2) = \int |f|^2\, \dd m$ for all $f\in C(X)$. Given this,
we can now follow the proof of Theorem \ref{t:converse} to conclude
the existence of (pairwise orthogonal) eigenfunctions  $e_\xi$ to
$\xi \in \Eig(X,G,m)$ with $\|e_\xi\|\leq 1$ and
\[
\sum_{\xi\in E} \left| \int_X f\, \overline{e_\xi}\, \dd m\right|^2
= \int_X |f|^2\, \dd m
\]
for all $f\in D$. As
$D$ is dense, this is only possible if $\|e_\xi\|=1$ holds for all
$\xi \in E$ and the $e_\xi$, $\xi \in E$, are an orthonormal basis.
This finishes the proof of the claim.

\medskip

Given the claim, the desired measurability follows easily: By
compactness and metrizability  of $X$ we can choose a countable dense
subset $D$ of $C(X)$. Then, the claim gives that $p\in X$ belongs to
$\Bap (X,G,m)$ if countably many conditions are satisfied. Clearly,
each of these conditions gives a measurable set.
\end{proof}

The following theorem can be seen as both a converse to Theorem
\ref{t:converse}  and an analogue to Theorem \ref{t:main}.

\begin{theorem}[Discrete spectrum via Besicovitch almost periodic
points]\label{t:main:bap} Let $(X,G,m)$ be  an ergodic dynamical
system and $(B_n)$ a F\o lner sequence along which Birkhoff's
ergodic theorem holds. Then, the following assertions are
equivalent:
\begin{itemize}
\item[(i)] The dynamical system $(X,G,m)$ has pure point spectrum.
\item[(ii)] $m(\Bap(X,G,m)) =1$.
\item[(iii)] $\Bap (X,G,m)\neq \varnothing$.
\end{itemize}
If one of the equivalent conditions (i), (ii) and (iii) holds, then
 $\Eig(X,G,m) = \Freq(x)$ for every $x\in \Bap(X,G,m)$.
Moreover, in this case, for any $f\in C(X)$ and $\xi \in
\widehat{G}$ the function
\begin{equation*}
e_{f,\xi} : X\longrightarrow \CC, \qquad e_{f,\xi}(x)
:=\left\{\begin{array}{r@{\quad:\quad}l}
 \Average(f_x \overline{\xi})
  & x\in \Bap (X,G,m)\\
 0 & \mbox{ else }
 \end{array}\right.,
\end{equation*}
satisfies  $P_\xi f = e_{f,\xi}$ (in $L^2 (X,m)$),  $e_{f,\xi} (tx)
= \xi(t) e_{f,\xi}(x)$ for all $t\in G$ and $x\in X$ and has
constant modulus on $\Bap (X,G,m)$.
\end{theorem}

\begin{remark} By Theorem \ref{t:converse}, the existence of a generic
Besicovitch almost periodic point entails the ergodicity of $m$. For
this reason the ergodicity assumption in the above theorem can not
be dropped.
\end{remark}

\begin{proof} The implication (ii)$\Longrightarrow$(iii) is
clear.  The implication (iii)$\Longrightarrow$(i) follows from
Theorem \ref{t:converse}. We now  show (i)$\Longrightarrow$(ii). As
the set of generic points has full measure, it suffices to show that
almost every $x\in X$ is Besicovitch almost periodic. To do so, we
denote the inner product on $L^2 (X,m)$ by
$\langle\cdot,\cdot\rangle$ and the associated norm by
$\|\cdot\|_2$. Let $\xi_1,\xi_2,\xi_3,\ldots, $ be an enumeration of
$\Eig(X,G,m)$.  Choose for any $\xi\in E$ a normalized eigenfunction
$e_\xi : X\longrightarrow \CC$. Without loss of generality, we can
assume
\[
e_\xi (sx) = \xi (s)\, e_\xi (x)
\]
for all $s\in G$ and $x\in X$. (Otherwise we could replace $e_\xi$
by $\widetilde{e}_\xi$ defined by
\[
\widetilde{e}_\xi(x) :=\lim_{n\to \infty} \frac{1}{|B_n|} \int_{B_n}
e(sx)\, \overline{\xi}(s)\, \dd s
\]
if the limit exists and
$\widetilde{e}_\xi(x) =0$ else.) Moreover, for $\xi = 1$ we choose
the constant function $1$.

Consider now an arbitrary  $g\in C(X)$. By  Birkhoff's ergodic
theorem, we can then  find a subset $X_g$ of $X$ of full measure
such that
\[
\int_X \left|g - \sum_{j=1}^k \langle g,e_{\xi_j} \rangle
e_{\xi_j}\right|\, \dd m(x)  = \lim_{n\to \infty}\frac{1}{|B_n|}
\int_{B_n} \left|g(sx) - \sum_{j=1}^k \langle g,e_{\xi_j}\rangle
e_{\xi_j} (sx)\right|\, \dd s
\]
for all $x\in X_g$ and $k\in \NN$. Let $D\subset C(X)$ be a
countable dense subset. Define
$$X':=\bigcap_{g\in D} X_g.$$
Then, $X'$ has full measure and a short computation gives
\[
\int_X \left|f - \sum_{j=1}^k \langle f,e_{\xi_j}\rangle
e_{\xi_j}\right|\, \dd m(x)  = \lim_{n\to \infty}\frac{1}{|B_n|}
\int_{B_n} \left|f(sx) - \sum_{j=1}^k \langle f,e_{\xi_j}\rangle
e_{\xi_j} (sx)\right|\, \dd s
\]
for all $f\in C(X)$, $x\in X'$ and $k\in \NN$. This, in turn
implies that any $x\in X'$ is Besicovitch almost periodic: Indeed, a
short calculation invoking Birkhoff's ergodic theorem and
Cauchy--Schwarz' inequality shows
\begin{eqnarray*}
\overline{M}( |f(\cdot x) -\sum_{j=1}^k \langle f, e_{\xi_j}
\rangle\, e_{\xi_j}(x)\, \xi_j (\cdot)|) &=& \overline{M}( |f(\cdot
x) -\sum_{j=1}^k \langle f, e_{\xi_j} \rangle\, e_{\xi_j}(\cdot
x)|)\\
(\mbox{Birkhoff's ergodic theorem})\;\: &=&\lim_{n\to
\infty}\frac{1}{|B_n|} \int_{B_n} \left|f(sx) - \sum_{j=1}^k \langle
f,e_{\xi_j}\rangle
e_{\xi_j} (sx)\right|\, \dd s\\
&=&\int_X \left|f - \sum_{j=1}^k \langle f,e_{\xi_j}\rangle
e_{\xi_j}\right|\, \dd m(x)\\
(\mbox{Cauchy--Schwarz' inequality})\;\: & \leq & \|f -\sum_{j=1}^k
\langle f,e_{\xi_j}\rangle e_{\xi_j}\|_2\\ &\to& 0, \;\: n\to
\infty.
\end{eqnarray*}
This gives the desired claim.

\smallskip

We now turn to the remaining statements. The equality $\Freq (x) =
\Eig(X,G,m)$ for an element $x\in \Bap (X,G,m)$ directly follows from Theorem
\ref{t:converse}. As for $e_{f,\xi}$ we note that it is well-defined
and invariant (as $\Bap(X,G,m)$ is invariant). The equality $\Freq
(x) = \Eig(X,G,m)$ for $x\in \Bap (X,G,m)$ rather directly gives
that  $e_{f,\xi}$ vanishes identically for $\xi
\in\widehat{G}\setminus \Eig(X,G,m)$. In particular, it has constant
modulus on $\Bap(X,G,m)$. Now, by Theorem \ref{t:converse},  for
each $\xi \in E$ and  $x\in \Bap (X,G,m)$ there exists a normalized
eigenfunction $e_\xi^{(x)}$ with
\[
e_{f,\xi} (x) =\Average(f_x \overline{\xi}) = \langle f, e_{\xi}^{(x)}\rangle
\]
As  each eigenspace is one-dimensional the $e_\xi^{(x)}$ arising for
different $x\in \Bap(X,G,m)$ will only differ by a factor of modulus
one. This gives the statement on constancy of the modulus.

That $e_{f,\xi}$ is the projection onto the eigenspace of $\xi$
follows from standard theory, see e.g. \cite{Len0} for recent
discussion.
\end{proof}

\section{Weyl almost periodic points, unique ergodicity  and continuity of
eigenfunctions}\label{s:weyl}

In this section, we consider a strengthening of Besicovitch almost
periodicity viz Weyl almost periodicity. We show that Weyl almost
periodicity extends from one point to its  orbit closure. This
allows us to characterize transitive systems all of whose points are
Weyl almost periodic. These are exactly the uniquely ergodic
dynamical systems with pure point spectrum and continuous
eigenfunctions. This ties in with various recent investigations (see
below for details).

\bigskip

Let    $(X,G)$ be a dynamical system. Whenever $d$ is a metric on
$X$ generating the topology and $(B_n)$ is a F\o lner sequence, we
define for each $n\in \NN$  the map
\[
\overline{M}_n :B(G) \longrightarrow \RR,\qquad \overline{M}_n
(f):=\sup_{s\in G} \frac{1}{|B_n|} \int_{B_n +s} f(t)\, \dd t.
\]
This gives then rise to the functions
\[
D_n :=D_{n,d} : X\times X\longrightarrow [0,\infty),\qquad D_n
(x,y):=\overline{M}_n (s\mapsto d(sx,sy)),
\]
for each $n\in\NN$. Then, each
$D_n$ can easily be seen to be an invariant pseudometric.
Moreover, for each $x\in X$ the function  $t\mapsto D_n (x,tx)$ is
uniformly continuous (by the argument used in Section \ref{s:mean}
to show uniform continuity of $D$). The function $D_n$ is referred
to as \textit{averaged  pseudometric} on level $n$.

A bounded measurable function $f : G\longrightarrow \CC$ is \textit{Weyl
almost periodic} if  for each $\varepsilon
>0$ there exist  $k\in \NN$, $\xi_1,\ldots, \xi_k \in \widehat{G}$
and $c_1,\ldots ,c_k\in\CC$ with
$$\limsup_{n\to \infty} \overline{M}_n (|f - \sum_{j=1}^k c_{\xi_j} \xi_j|) <
\varepsilon.$$ As discussed in Appendix \ref{s:a:wap}, an equivalent
alternative characterization is that for each $\varepsilon
>0$ there exist $N\in\NN$ and a relatively dense set $R\subset G$
with
$$ \overline{M}_N (|f - f(\cdot - t)|) <
\varepsilon$$ for all $t\in R$. A crucial feature of Weyl almost
periodic functions is the  existence of the limits
\[
\lim_{n\to\infty} \frac{1}{|B_n|} \int_{B_n + s_n} f(t)\, \xi (t)\,
\dd t
\]
irrespective of (and uniform in) the chosen sequence $(s_n)\in G$
for each $\xi \in \widehat{G}$, see Appendix \ref{s:a:wap}.

\begin{definition}[Weyl almost periodic points]
 Let $(X,G)$ be a dynamical system and $d$ a
metric on $X$ generating the topology and $(B_n)$ a F\o lner
sequence and $D_n$, $n\in\NN$,  the associated averaged metrics.
Then, a point $x\in X$ is called  Weyl  almost periodic with respect
to $d$ and $(B_n)$ if for every $\varepsilon >0$ there exists an
$N\in\NN$ such that
$$\{t\in G: D_N (x,tx)<\varepsilon\}$$
is relatively dense.
\end{definition}
\begin{remark} It follows from Proposition \ref{p:a:final-averaging}
that an $x\in X$ is Weyl almost periodic if and only if for each
$\varepsilon >0$ there exists a relatively dense set $R\subset G$
and an $N_0\in\NN$ such that $D_N (x,tx)<\varepsilon$ for all $N\geq
N_0$ and $t\in R$.
\end{remark}

Arguing as in Section \ref{s:main} with $\overline{M}$ replaced by
$\overline{M}_n$ we see that Weyl almost periodicity is independent
of the chosen metric and the following holds.

\begin{lemma} Let  $(X,G)$ be a dynamical system and $d$ a
metric on $X$ generating the topology and $(B_n)$ a F\o lner
sequence and $D_n$, $n\in\NN$,  the associated averaged metrics.
Then, the following assertions for $x\in X$ are equivalent:
\begin{itemize}
\item[(i)] The point $x$ is Weyl almost periodic.
\item[(ii)] The algebra $\mathcal{A}_x$ consists of Weyl almost
periodic functions.
\item[(iii)] The set $\{f\in C(X):  f_x \mbox{ is Weyl almost
periodic} \}$ separates the points of $X$.

\item[(iv)]  For any $s\in G$ the function $d^{(s)}_x$
is Weyl  almost periodic, where $d^{(s)} \in C(X)$ is defined via
$d^{(s)}(y) : =d(sx,y)$.

\end{itemize}
\end{lemma}
The previous lemma implies in particular that any Weyl almost
periodic point is Besicovitch almost periodic. It is not hard to see
by examples that the converse does not hold.

\smallskip

Weyl almost periodicity has a stability property.

\begin{lemma}[Stability of Weyl almost periodicity along orbit
closures]\label{l:orbit-closure}
 Let $(X,G)$ be a dynamical system. Assume that $x\in X$ is  Weyl
almost periodic. Then, any element in the orbit closure of $x$  is
Weyl almost periodic.
\end{lemma}
\begin{proof}  The function $D_N$ is lower semicontinuous for each $N\in \NN$ as it is  a
supremum over continuous functions. From this and the invariance of
$D_N$ we find
$$D_N (y,ty)\leq \liminf_{n\to\infty} D_N (s_n x, t s_n x) = D_N (x,tx)$$
whenever $s_n x\to y$ for a sequence $(s_n)$ in $G$. This easily
gives the desired statement.
\end{proof}

\begin{prop}\label{p:one-half} Let $(X,G)$ be a  dynamical system with
transitive element $p\in X$. Let  $p$ be Weyl almost periodic. Then
$(X,G)$ is uniquely ergodic, has pure point spectrum,  all
eigenfunctions are continuous and  $\Freq(x) = \Eig(X,G,m)$ holds
for all $x\in X$. Moreover, for any  $f\in C(X)$ and $\xi \in
\widehat{G}$, the averages
\[
A_n (f_x\overline{\xi}):=\frac{1}{B_n} \int_{B_n} f(tx)\,
\overline{\xi(t)}\, \dd t
\]
converge (uniformly in $x$) towards the projection of $f$ onto the
eigenspace of $\xi$.
\end{prop}
\begin{proof} It is well-known that unique ergodicity is equivalent
to uniform (in $y\in X$) convergence  of the averages
\[
\frac{1}{|B_n|} \int_{B_n} f(ty)\, \dd t
\]
for each  continuous  $f:X\longrightarrow \CC$. Now,  uniform
existence of these averages  on the orbit of $x$  is a direct
consequence of Weyl almost periodicity. This easily gives uniform
existence on the orbit closure. As the orbit closure is $X$ the
desired statement on unique ergodicity follows. Denote the unique
invariant probability measure by $m$.

By the previous lemma and the transitivity  assumption on $p$ every
$x\in X$ is Weyl almost periodic. In particular, every element is
Besicovitch almost periodic. As $(X,G)$ is uniquely ergodic every
$x\in X$ is also generic with respect to $m$. Hence, $X=\Bap(X,G,m)$
follows. By Theorem \ref{t:main:bap},  this implies pure point
spectrum as well as pointwise convergence of the averages $A_n
(f_x\overline{\xi})$ to the projection of $f$ onto the eigenspace of
$\xi$ for each $f\in C(X)$ and $\xi \in \widehat{G}$. Now, by Weyl
almost periodicity these averages converge uniformly in $x\in X$.
Hence, their limit is continuous and continuity of the
eigenfunctions follows.
\end{proof}

\begin{theorem}\label{t:main-weyl}
Let $(X,G)$ be a dynamical system with transitive point  $p\in X$.
Then, the following assertions are equivalent:
\begin{itemize}
\item[(i)] The dynamical system $(X,G)$ is uniquely ergodic with
pure point spectrum and continuous eigenfunctions.
\item[(ii)] The point $p$ is Weyl almost periodic.
\item[(iii)] Every $x \in X$ is Weyl almost periodic.
\end{itemize}
In this case, we have $\Freq(x) = \Eig(X,G,m)$ for all $x\in X$.
\end{theorem}
\begin{proof}
 The implication (iii)$\Longrightarrow$(ii) is obvious while
(ii)$\Longrightarrow$(iii) follows from Lemma~\ref{l:orbit-closure}.

The implication (ii)$\Longrightarrow$(i) and the last part of the
theorem were shown in Proposition \ref{p:one-half}. It remains to
show the reverse implication (i)$\Longrightarrow$(ii): This follows
by a variant of the proof of the corresponding part in Theorem
\ref{t:main:bap}. We denote the unique invariant measure by $m$ and
use the notation introduced in the proof of Theorem
\ref{t:main:bap}. Thus, we denote the inner product on $L^2 (X,m)$
by $\langle\cdot,\cdot\rangle$ and the associated norm by
$\|\cdot\|_2$.  As the spectrum is pure point, there exists an
orthonormal basis $e_\xi$, $\xi\in \Eig(X,G,m)$, of $L^2 (X,m)$ with
$e_\xi$ being an eigenfunction to the eigenvalue $\xi$ for each
$\xi\in \Eig(X,G,m)$. By assumption each $e_\xi$, $\xi\in E$, can be
chosen continuous. By unique ergodicity we then find for any finite
subset $A\subset \Eig(X,G,m)$ and any $y\in X$:
\begin{eqnarray*}\limsup_{n\to\infty} \overline{M}_n(|f(\cdot y)  - \sum_{\xi \in A} \langle e_{\xi},
f\rangle e_{\xi} (y) \xi (\cdot)|)
    &=& \limsup_{n\to\infty} \overline{M}_n
|f(\cdot y) - \sum_{\xi \in A} \langle e_{\xi}, f\rangle e_{\xi}
(\cdot y)|\\
\;\:
    & =& \lim_{n\to\infty}\frac{1}{|B_n|} \int_{B_n} |f(ty)
-\sum_{\xi \in A} \langle e_{\xi}, f\rangle
e_{\xi} (t y)|\, \dd t\\
    & =&  \int_X | f(x) -\sum_{\xi \in A} \langle e_{\xi}, f\rangle
e_{\xi} ( x)|\, \dd m(x)\\
(\mbox{Cauchy--Schwarz' inequality})\;\:
    &\leq& \|f  -\sum_{\xi \in A}
\langle e_{\xi}, f\rangle e_{\xi} \|_2.
\end{eqnarray*}
As $f_\xi$, $\xi \in E$, is a basis of  $L^2 (X,m)$, the last term
becomes arbitrarily small for suitable $A\subset E$.  This  shows
that $t\mapsto f(ty)$ is Weyl almost periodic for any $y\in X$.
\end{proof}

\begin{remark}
Recently, systems satisfying the equivalent conditions of the
theorem have attracted substantial interest:

(a) For  $G = \ZZ$ various equivalent characterizations of (i) have
been investigated in \cite{DG}. In particular, it is  shown there
that (i) is equivalent to topological isomorphy of the system to its
maximal equicontinuous factor.  The case of general amenable groups
$G$  has been studied in \cite{FGL}. In particular, it has been
shown there that (i) is equivalent to the continuity of the averaged
metric $D$ on $X\times X$. This continuity is known as \textit{mean
equicontinuity} of the system. In this context our preceding result
is remarkable as it does not assume control of $D$ on the whole of
$X\times X$ but just on $\{(x,tx) : t\in G\}$. In this sense,  we
have found a pointwise characterization of mean equicontinuity.
Note, however, that we require a rather uniform control on the orbit
of this one point.

(b) A large class of examples satisfying the conditions of the
theorem are  weakly almost periodic systems. A recent study of such
systems is carried out  in \cite{LS} to which we refer for a precise
definition and further references, see also subsequent Section
\ref{s:Weakly-and-Bohr}.

(c) A most important class of examples for the theorem are dynamical
systems arising from regular cut and project schemes. Such systems
are at the core of the study of aperiodic order (see \cite{TAO}).
They belong to the special class of dynamical systems known as
translation bounded measure dynamical system (TMDS), see Section
\ref{s:application} for details. In fact, a  huge bulk of material
in the theory of aperiodic order deals with TMDS  satisfying (i) of
the theorem. A characterization of such systems via an almost
periodicity property of its points had been missing for a long time.
It was only given recently in \cite{LeSpStr}. The preceding theorem
is a generalization of the corresponding result of \cite{LeSpStr} in
that it is not restricted to TMDS but rather applies to general
dynamical systems.
\end{remark}

If a system is minimal every point is transitive and we can  note
the following immediate consequence of the preceding theorem.

\begin{coro}
Let $(X,G)$ be a minimal  dynamical system. Then, the following
assertions are equivalent:
\begin{itemize}
\item[(i)] Every point in $X$  is Weyl almost periodic.
\item[(ii)] There exists a Weyl almost periodic point in $X$.
\item[(iii)] The dynamical system $(X,G)$ is uniquely ergodic with
pure point spectrum and continuous eigenfunctions.
\end{itemize}
\end{coro}
\begin{remark} A system can well be Weyl almost periodic without
being minimal. Consider for example the orbit closure in the
subshift $(\{0,1\}^\ZZ,\ZZ)$  of the element $\omega$ defined with
$\omega_0 = 1$ and $\omega_n = 0$ for $n\neq 0$.
\end{remark}

\section{Weakly and Bohr almost periodic dynamical
systems}\label{s:Weakly-and-Bohr} In the preceding section, we have
met Weyl almost periodic points. In this section, we introduce two
special classes of Weyl almost periodic points, viz Bohr almost
periodic points and weakly almost periodic points. This allows us to
reanalyze and characterize weakly almost periodic dynamical system
and Bohr almost periodic dynamical system  via the new approach in
this paper. Specifically, the main result of this section shows that
a (transitive)  dynamical system is weakly almost periodic (Bohr
almost periodic) if and only if every of its points is weakly almost
periodic (Bohr almost periodic). We refer to \cite{LS2} for a recent
discussion of Bohr and weakly almost periodic systems including
relevance, background and further references.

\bigskip

Let $G$ be a {$\sigma$-compact} locally compact abelian group and
denote the set of uniformly continuous and bounded functions on $G$ by $C_u
(G)$. This space is a Banach space when equipped with the supremum
norm $\|\cdot\|_\infty$. An $f\in C_u (G)$ is called \textit{weakly
almost periodic} if the set $\{f(\cdot -t) : t\in G\}$ is relatively
compact in $C_u (G)$ with respect to the weak topology of the Banach
space $(C_u (G),\|\cdot\|_\infty)$. Clearly, any Bohr almost
periodic $f\in C_u (G)$  is weakly almost periodic (as Bohr almost
periodicity means by definition that the set $\{f(\cdot -t) : t\in
G\}$  is relatively compact in the original topology). In fact, a
main result on weakly almost periodic functions, see e.g. \cite{AG}
or \cite{EN,LS2}, gives that any weakly almost periodic $f$ can be
(uniquely) decomposed into $f = g + h$ with $g\in C_u (G) $ Bohr
almost periodic and $h\in C_u (G)$ satisfying
$$
\lim_{n\to\infty} \sup_{s\in G} \frac{1}{|B_n|} \int_{B_n} |h(s + t) |\, \dd t  = 0
$$
for any F \o lner
sequence $(B_n)$. {The existence of this decomposition is
\cite[Theorem~4.7.11]{MoSt}, while the uniqueness follows
immediately from \cite[Lemma~4.6.8]{MoSt}.}

When combined with $(\heartsuit)$ in Appendix
\ref{s:a:bap} this easily gives that any weakly almost periodic
function is Weyl almost periodic.

\begin{definition}[Weakly and Bohr almost periodic points] Let $(X,G)$ be a
dynamical system.

(a) A point $x\in X$ is called weakly almost periodic if
$\mathcal{A}_x$ consists only of weakly almost periodic functions.

(b) A point $x\in X$ is called Bohr almost periodic if
$\mathcal{A}_x$ consists only of Bohr almost periodic functions.

\end{definition}

\begin{remark} By the discussion preceding the definition, any
Bohr almost periodic point is weakly almost periodic and any weakly
almost periodic point is Weyl almost periodic.
\end{remark}

In the subsequent discussion, the weakly almost periodic case and the
Bohr almost periodic case can be mostly treated in parallel. To
facilitate the reading, we then give statements for the weakly almost
periodic case and mention the Bohr almost periodic case in brackets
only.

With  essentially the same proof as Lemma \ref{l:char-ap-functions},
we obtain the following statement.

\begin{lemma}\label{l:metric-char-bap-wap} Let $(X,G)$ be a dynamical system and let $x \in X$ be given.
Then, the following assertions are equivalent:
\begin{itemize}
  \item [(i)] The point  $x$ is weakly (Bohr) almost periodic.
  \item [(ii)] The weakly (Bohr) almost periodic functions in $\cA_x$ are dense in $(\cA_x, \| \cdot \|_\infty)$.
  \item [(iii)] The set  $\{ f \in C(X): f_x  \mbox{ is weakly (Bohr) almost periodic } \}$
separates the points of $X$.
\item[(iv)] For any $s\in G$ the function $d^{(s)}_x$ is weakly (Bohr) almost
periodic,  where $d^{(s)} \in C(X)$ is defined via $d^{(s)}(y) :
=d(sx,y)$.
\end{itemize}
\end{lemma}

Whenever $(X,G)$ is a dynamical system, any $f\in C(X)$ gives rise
to a function on the product $G\times X$, viz
$$
P_f : G\times X \longrightarrow \CC,\qquad (t,x) \mapsto
f(tx) \,.
$$  Roughly speaking one can say that the  preceding
discussion was concerned with almost periodicity properties of  the
functions $f_x = P_f (\cdot,x)$ for $x\in X$. It is natural to
consider almost periodicity properties of  restrictions of the $P_f$
to $X$ as well for fixed $t\in G$. This leads to the notion of
weakly (Bohr) almost periodic dynamical system.  For $f\in C(X)$ and
$t\in G$ we define
$$
f_t : X\longrightarrow \CC,\qquad f_t (x) = f(tx) = P_f (t,\cdot) \,.
$$

\begin{definition}[Weakly and Bohr almost periodic dynamical systems]
The dynamical system  $(X, G)$ is called weakly almost periodic and
Bohr almost periodic\footnote{In some papers this is called almost
periodic dynamical system, or strongly almost periodic dynamical
system.} respectively  if for any $f\in C(X)$ the family $\{f_t\, :\, t\in G\}$ has compact closure in the weak topology and the Banach
space topology of $(C(X),\|\cdot\|_\infty)$ respectively.
\end{definition}

The next  lemma relates weakly (Bohr) almost periodic dynamical
systems to  weakly (Bohr) almost periodic points.

\begin{lemma} \label{Lemma WAP DS implies WAP elements} Let $(X,G)$
be a dynamical system.

(a) If $(X,G)$ is weakly (Bohr) almost periodic then every $x \in X$
is a weakly (Bohr) almost periodic.

(b) If $x \in X$ is  transitive and weakly (Bohr) almost periodic
then $(X,G)$ is a weakly (Bohr) almost periodic dynamical system.

\end{lemma}
\begin{proof}
Fix an arbitrary $x \in X$. Define $F: C(X) \to C_u(G)$ via
\begin{displaymath}
F(f)(t)=f_x\,.
\end{displaymath}

(a) It is easy to see that $F$ is well defined and $\| F \| \leq 1$.
It follows that $F$ is continuous, and hence also weakly continuous
\cite[Lemma.~4.4.2]{MoSt}. Therefore, the image of a compact,
respectively weak compact set, is compact, respectively weak
compact. Since $F$ commutes with the group action, the claim
follows.

\smallskip

(b) We know that $F$ is continuous. Moreover, since $x$ has dense
orbit, it follows immediately that for all $f \in C(X)$ we have
\begin{displaymath}
\| F(f) \|_\infty = \| f \|_\infty \,.
\end{displaymath}
Therefore, $F$ is an isometry and hence it induces an isomorphic
isometry $F: C(X) \to \mbox{Im}(F)$. In particular, $\mbox{Im}(F)$
is closed in $C_u(G)$ and the mapping
\begin{displaymath}
F^{-1} : \mbox{Im}(F) \to C(X) \,,
\end{displaymath}
is a continuous operator, and hence is also weakly continuous
\cite[Lemma.~4.4.2]{MoSt}. Therefore, $F^{-1}$ maps compact and
weakly compact sets, respectively, into compact and weakly compact
sets respectively. This easily gives the desired statement.
\end{proof}

We obtain the following immediate consequence of the preceding
lemma.

\begin{theorem}\label{thm weakl ap ds}
Let $(X,G)$ be a dynamical system with transitive  $p \in X$. Then,
the  following assertions are equivalent:
\begin{itemize}
  \item [(i)] $(X,G)$ is weakly almost periodic.
  \item [(ii)] Every $x \in X$ is  weakly almost periodic.
  \item [(iii)] The point  $p\in X$ is  weakly almost periodic.
\end{itemize}
\end{theorem}

We also note the following consequence of the theorem.

\begin{coro}\label{c:weakly_ap_ds}
Let $(X,G)$ be a dynamical system with transitive $p\in X$. For $s,t
\in G$ define $g_{s,t}: X \to \CC$ via
\begin{displaymath}
g_{s,t}(x)=d(sp,tx) \,.
\end{displaymath}
Then,  $p\in X$ is weakly  almost periodic if and only if for each
$s \in G$ the set $\{ g_{s,t} : t \in G \}$ has compact closure in
the weak  topology of $(C(X),\|\cdot\|_\infty)$.
\end{coro}
\begin{proof} Indeed, the only if
part of the statement is immediate from the definition of   weak
almost periodicity and the preceding theorem. To show the if part we
will denote (in line with \cite{LS2})
\[
\WAP(X):= \{ f \in C(X) : \{ f_t :t \in G \} \mbox{ has compact
closure in the weak  topology} \}\,.
\]
Then, by \cite[Prop.~3.3,Prop 3.4]{LS2} $\WAP(X)$ is a closed algebra
of $(C(X), \| \cdot \|_\infty)$, which contains the constant function 1.
Moreover,  we have $g_{t,1} \in \WAP(X)$ for all $t\in G$. We next
show  that the functions $g_{t,1}$ separate the points of $X$. Let
$x \neq y \in X$ be arbitrary, and let $r=d(x,y)$.

Since $p$ is a transitive point, there exists some $t \in G$ such
that $d(tp,x) < \frac{r}{3}$. Then,
\[
  g_{t,1}(x) =d(tp,x) < \frac{r}{3} \qquad \text{ and } \qquad
  g_{t,1}(y)= d(tp,y) \geq d(x,y) -d (tp,x) > \frac{2r}{3} \,.
\]
This shows that $g_{t,1}(x) \neq g_{t,1}(y)$.
Therefore, $\WAP(X)$ is a closed algebra of $(C(X), \| \cdot \|_\infty)$
which is separating the points, and hence, by Stone--Weierstra\ss{}' theorem,
$\WAP(X)=C(X)$. This gives the desired statement.
\end{proof}

The preceding theorem allows us to recapture a main result of
\cite{LS2}.

\begin{coro}
Let $(X,G)$ be a transitive weakly almost periodic dynamical system.
Then, $(X,G)$ is uniquely ergodic with pure point spectrum and
continuous eigenfunctions.
\end{coro}
\begin{proof}
Let $p$ be a transitive point. Then, by Theorem~\ref{thm weakl ap
ds}, $p$ is weakly almost periodic. Hence, $p$ is Weyl almost
periodic as well and the claim follows from
Theorem~\ref{t:main-weyl}.
\end{proof}

As should be clear from the preceding discussion, the analogue of
Theorem \ref{thm weakl ap ds} with `weakly almost periodic'
replaced by `Bohr almost periodic' holds as well (and similarly for
Corollary \ref{c:weakly_ap_ds}). In fact, a somewhat stronger
statement is true for Bohr almost periodic points. To state this
properly, we introduce the following notation when dealing with an
dynamical system $(X,G)$ with metric $d$. For $f\in C(X)$ we consider the mapping
$$
\pi_f : X\longrightarrow C_u (G), \qquad \pi_f (x) := f_x \,,
$$
and, when an $x\in X$ is fixed, we set
$$
Y(f,x):=\overline{ \{f_{tx} : t\in G\} }  \, ,
$$
where the closure is taken in $C_u (G)$ with the
(usual) supremum norm. We also define
$$
\overlined : X\times X \longrightarrow [0,\infty), \qquad \overlined(x,y):=
\sup_{s\in G} d(sx,sy) \,.
$$
Clearly, $\overlined$ is a metric.

\begin{theorem}\label{t:char-bap-point} Let $(X,G)$ be a dynamical system. For $x\in X$ the
following assertions are equivalent:
\begin{itemize}
  \item [(i)] The element $x\in X$ is Bohr almost periodic.
\item[(ii)] For any $f\in C(X)$  the map
$\pi_f : \overline{Gx} \longrightarrow C_u (G)$, $y\mapsto f_y$, is
continuous with range given by   $Y(f,x)$.
\item[(iii)] There exists a $G$-invariant metric on $\overline{Gx}$ generating the topology.
 \item[(iv)] The function $\overlined$ is continuous on
$\overline{Gx}$.
\item[(v)] The orbit closure $\overline{Gx}$  admits a structure of a locally compact group such that $G\longrightarrow
  \overline{Gx}$,  $t\mapsto t x$,  becomes a continuous group homomorphism.
\end{itemize}
In particular, the orbit closure of any Bohr almost periodic point
is minimal.
\end{theorem}
\begin{proof}
(i)$\Longrightarrow$(ii): Clearly, $\pi_f (tx) = f_{tx} = f_x
(\cdot+t)$ for any $t\in G$. Thus, it suffices to  show that $f_y$
belongs indeed to $Y(f,x)$ for any $y\in \overline{Gx}$ and $\pi_f$
is continuous on $Y(f,x)$. It is enough to show that $\pi_f(y_n)$
converges to $\pi_f (y)$ whenever $(y_n)$ is a sequence in
$\overline{Gx}$ converging to $y\in \overline{Gx}$ such that $\pi_f
(y_n)$ belong to $Y(f,x)$. Now, it is not hard to see that the
functions $\pi_f (y_n)$ converge pointwise to $\pi_f (y)$. Moreover,
by (i), the set $Y(f,x)$ is compact and, hence, $\pi_f(y_n)$ has a
uniform convergent subsequence. Now, any such subsequence must
converge to $\pi_f (y)$ (as uniform convergence implies pointwise
convergence). This gives the desired convergence statement.

\smallskip

(ii)$\Longrightarrow$(iii): By (ii) the map  $\overlined_f$ with
$\overlined_f (y,z) := \|\pi_f (y) - \pi_f (z)\|_\infty$ is a
continuous pseudometric on $\overline{Gx}$. It is clearly invariant.
Now, choose a countable dense subset $D\subset C(X)$ separating the
points of $\overline{Gx}$. Assume without loss of generality that
any element of $D$ is normalized, and choose for any $f\in D$ a $c_f
>0$ with $\sum_{f\in D} c_f <\infty$. Then, $\sum_{f\in D} c_f
\overlined_f$ is a continuous invariant metric on $\overline{Gx}$.
As $\overline{Gx}$ is compact any continuous metric determines its
topology.

\smallskip

(iii)$\Longrightarrow$(iv): Let $\overlined'$ be a continuous
invariant metric on $\overline{Gx}$. Let $\varepsilon >0$ be
arbitrary. As $\overlined'$ is a continuous metric on
$\overline{Gx}$ there exists a $\delta>0$ with $d(y_1,y_2)\leq
\varepsilon$ for $y_1,y_2\in \overline{Gx}$  whenever
$\overline{d}'(y_1,y_2)\leq \delta$. As $\overlined'$ is invariant,
we obtain then $d(t y_1, t y_2)\leq \varepsilon$ for all $t\in G$
and, hence, $\overlined(y_1,y_2)\leq \varepsilon$   whenever
$\overlined'(y_1,y_2)\leq \delta$ for $y_1,y_2\in \overline{Gx}$.
This is the desired statement.

\smallskip

(iv)$\Longrightarrow$(v): Using the invariant metric $\overlined$ it
is not hard to see that there is a group structure on
$\overline{Gx}$ with $tx + sx = (t+s)x$. Here, we show only that
this is well-defined. The remaining statements then follow easily.
Assume $tx = t'x$ and $sx = s'x$. Then triangle inequality  and
invariance of the metric gives
$$\overlined((t+s)x,(t' + s')x)\leq  \overlined((t+s)x,(t+s')x) +\overlined((t+s')x,
(t' + s')x) \leq \overlined(sx,s'x) + \overlined(tx,t'x) = 0.$$ This
shows well-definedness.

\smallskip

(v)$\Longrightarrow$(i): This is standard. We include some details
for convenience of the reader. Let $f$ be a continuous function on
$X$. We have to show that the set  $S:=\{f_x (t + \cdot): t\in G\} =
\{f_{tx} : t\in G\}$ has compact closure in $C_u (G)$ with respect
to the supremum norm. From (iv) we easily see that $\pi_f :
\overline{Gx}\longrightarrow C_u (G), y\mapsto f_y$, is continuous.
Hence,  $\pi_f (\overline{Gx})$ is compact and, as it  clearly
contains $S$, the desired statement follows.

\smallskip

The minimality statement follows directly from (iii).
\end{proof}

The preceding result shows that  Bohr almost periodic points give
rise to minimal orbit closures. In fact, within the weakly almost
periodic points  one can even characterize the Bohr almost periodic
points by minimality of their orbit closures:

\begin{prop} Let $(X,G)$ be a dynamical system and $x\in X$ be
weakly almost periodic. Then, $x$ is Bohr almost periodic if and
only if its orbit closure $\overline{Gx}$ is minimal.
\end{prop}
\begin{proof} We have just seen in Theorem \ref{t:char-bap-point} that the orbit closure of a Bohr
almost periodic point is minimal. So, consider now a weakly almost
periodic point $x\in X$ with minimal orbit closure $\overline{Gx}$.
Then, clearly $(\overline{Gx},G)$ is  weakly almost periodic  by (b)
of Lemma \ref{Lemma WAP DS implies WAP elements} and it is minimal
(by assumption). Now, as is well known (see e.g. Prop.~3.7 of
\cite{LS2}) any minimal component of a weakly almost periodic system
is Bohr almost periodic. Now, the desired claim follows from (a) of
Lemma \ref{Lemma WAP DS implies WAP elements}\end{proof}

\begin{remark} We note that a Bohr almost periodic dynamical system
does not need to be minimal as can easily be seen by considering the
`disjoint union' of two Bohr almost periodic systems.
\end{remark}

It is possible to characterize Bohr almost periodic
points by almost periodicity properties of the metric $\overlined$.

\begin{theorem}\label{t:Bohr}
Let $(X,G)$ be a dynamical system with metric $d$. Then the following assertions are equivalent for a point $p\in X$:
\begin{itemize}
\item[(i)] The point $p$ is Bohr almost periodic.
\item[(iii)] The function $G\mapsto [0,\infty)$, $t\mapsto
\overlined(p,tp)$, is Bohr almost periodic.
\end{itemize}
\end{theorem}
\begin{proof}
We show that (i) implies (ii): By  (i) and Theorem
\ref{t:char-bap-point}, the orbit closure of $p$ is minimal and the
function $\overlined$ is a continuous metric on $\overline{Gp}$.
 Let $\varepsilon
>0$ be given.  As $\overlined$ is continuous and every Bohr almost
periodic point has a minimal orbit, there exists a relatively dense
set $R\subset G$ with $\overlined(sp,p)< \varepsilon$ for all $s\in
R$. As $\overlined$ is invariant, this gives
$$|\overlined(p,(t+s)p) - \overlined(p,tp)| \leq \overlined ((t+s)p,tp)) \leq
\overlined (sp,p)<\varepsilon$$ for all $t\in G$ and $s\in R$. As
$\varepsilon >0$ was arbitrary this gives (ii).

\smallskip

We now show that (ii) implies that $p$ is Bohr almost periodic: It
is not hard  to see that (ii) implies that $t \mapsto d(sp,(t+s),p)$
is Bohr almost periodic for any $s\in G$. This easily implies that
(iv) of Lemma \ref{l:metric-char-bap-wap} holds and (i) follows from
that lemma.
\end{proof}

\section{Application to measure dynamical
systems}\label{s:application} In this section, we use our results to
shed a light on recent investigations of aperiodic order. A
fundamental  issue in the study of aperiodic order is pure point
diffraction, see e.g. \cite{BL} for a recent survey. Indeed,
understanding of pure point diffraction has been a driving force in
the field, see e.g. the survey article \cite{Lag}. Recently, a
complete understanding of pure point diffraction  via mean almost
periodicity has been provided
 in \cite{LeSpStr}.  That article mostly
deals with single  measures. However, it also includes results on
pure point spectrum of certain dynamical systems viz measure
dynamical systems. Here, we discuss how our results allow one to
provide a different approach to these results.

\bigskip

We start with a discussion of translation bounded measure dynamical
systems. Such dynamical systems were brought forward in \cite{BL1}
to provide a systematic framework to study aperiodic order. In our
exposition we follow \cite{BL1} to which we refer for further
details, proofs and references.

We denote by $C_c (G)$ the vector space of continuous complex valued
functions on $G$ with compact support. This space is equipped with
the inductive limit topology of the injections
$$C_K (G)\longrightarrow C_c (G),\qquad \varphi \mapsto \varphi,$$
for $K\subset G$ compact. Here,  $C_K (G)$ denotes the  subspace of
$C_c (G)$ consisting of functions with support in $K$. The measures
on $G$ are the elements of the dual space of $C_c (G)$. The total
variation $|\mu|$ of a measure $\mu$ is the smallest positive
measure with
$$|\mu(\varphi)|\leq |\mu|(\varphi)$$
for all $\varphi \in C_c (G)$ with $\varphi \geq 0$. A measure $\mu$
on $G$ is called \textit{translation bounded} if its total variation
$|\mu|$ satisfies {
$$\| \mu \|_K:=\sup |\mu| (t + U) < \infty$$}
for one (all) relatively compact open $U$ in $G$. We denote that set
of all translation bounded measures by $M^\infty (G)$ and equip it
with the vague topology. Then, $G$ admits a natural action on
$M^\infty (G)$ by translations. More specifically, for $t\in G$ and
$\mu\in M^\infty (G)$ the measure $t\mu $ is defined by
$$t\mu (\varphi) = \mu (\varphi (\cdot + t))$$
 for all $\varphi \in C_c (G)$.

{ A subset $\varOmega \subset M^\infty (G)$ which is invariant under
the translation action is compact if and only if it is vaguely
closed and there exists a constant $C$ such that \cite[Thm.~A.8]{SS}
\begin{displaymath}
\| \mu \|_U \leq C \quad \text{ for all } \mu \in \varOmega \,.
\end{displaymath}}

Whenever  $\varOmega$ is a compact subset  of $M^\infty (G)$,  which
is invariant under the translation action to and $m$ is an
invariant probability measure on $X$, we call $(X,G,m)$ a
\textit{dynamical system of translation bounded measures} or just
TMDS for short. If $G$ is second countable than any TMDS is
metrizable. Hence, the theory developed above applies to TMDS
whenever $G$ is second countable.  Consider now an arbitrary TMDS
$(\varOmega,G,m)$ and define for any $\varphi \in C_c (G)$ the
function
$$N_\varphi : \varOmega \longrightarrow \CC, \qquad N_\varphi (\omega) = \omega (\varphi).$$
Then, $N_\varphi$ belongs to $C(\varOmega)$. Also, there exists  a
unique translation bounded measure $\gamma = \gamma^m$ on $(X,G,m)$
with
\begin{equation*}
\gamma (\varphi \ast \widetilde{\psi}) = \langle N_\varphi,
N_\psi\rangle
\end{equation*} for all $\varphi,\psi \in
C_c (G)$ and all $t\in G$.\footnote{Note that \cite{BL1} uses a
different sign in the definition of $\nfunction$ (called $f$ there)
as well as has the inner product linear in the second argument. This
results in a different display of the formula for $\gamma$, viz.
$(\gamma \ast \widetilde{\varphi}\ast \psi )(0) = \langle f_\varphi,
f_\psi\rangle$.}  The measure $\gamma$ is called the
\textit{autocorrelation} of the TMDS. This measure allows for a
Fourier transform $\widehat{\gamma}$ which is a (positive) measure
on $\widehat{G}$. It is known as \textit{diffraction} of the TMDS.
Of particular interest in this theory are now those TMDS whose
diffraction is a pure point measure. By a main result of \cite{BL1}
(see references there as well for earlier results) the diffraction
of a TMDS is pure point if and only if the TMDS has pure point
spectrum. So, for this reason, TMDS with pure point spectrum are of
utmost relevance in the field of aperiodic order. One particular
question is the calculation of the atoms of $\widehat{\gamma}$.
Here, the basic idea is that
\[
\widehat{\gamma} (\{\xi\}) = \lim_{n \to
\infty}\left|\frac{1}{|B_n|} \int_{B_n} \xi (t)\, \dd\omega
(t)\right|^2
\]
(with $(B_n)$ being a F\o lner sequence).  Validity
of this formula is often discussed under the heading of Bombieri--Taylor conjecture.

Having provided  the framework of TMDS, we  now discuss how the
theory developed in the previous section can be used in the study of
aperiodic order.

A translation bounded measure $\omega$ is called \textit{mean almost
periodic} (\textit{Besicovitch almost periodic}, \textit{Weyl almost periodic} respectively) if for any $\varphi \in C_c (G)$
the function
\[
\omega \ast \varphi : G\longrightarrow \CC,\qquad  (\omega \ast
\varphi )(t) = \int \varphi (t-s)\, \dd\omega (s),
\]
is mean almost periodic (Besicovitch almost periodic, Weyl almost periodic respectively).

\begin{prop}\label{p:observation}
Let $(\varOmega,G,m)$ be a TMDS and $(B_n)$ a F\o lner sequence.
Then, for $\omega\in\varOmega$ the following assertions are
equivalent:

\begin{itemize}
\item[(i)]  The measure $\omega$ is mean almost periodic
(Besicovitch almost periodic, Weyl almost periodic,  Weak almost
periodic, Bohr almost periodic).

\item[(ii)] The function $t\mapsto N_\varphi (t\omega) $ is mean almost periodic
(Besicovitch almost periodic, Weyl almost periodic, Weak almost
periodic, Bohr almost periodic)
 for any
$\varphi \in C_c (G)$.

\item[(iii)] The point $\omega\in \Omega$ is mean almost periodic (Besicovitch
almost periodic, Weyl almost periodic, Weak almost periodic, Bohr
almost periodic).

\end{itemize}

\end{prop}
\begin{proof}
We only discuss mean almost periodicity. The remaining
statements follow analogously.

(i)$\Longleftrightarrow$(ii):  A short computation shows
$$N_\varphi (t\omega) = (\omega \ast \widetilde{\varphi})(t),$$
where $\widetilde{\varphi} : G\longrightarrow \CC$, $t\mapsto
\overline{\varphi}(-t)$. This gives that $\omega$ is mean almost
periodic if and only if $t\mapsto N_\varphi (t\omega)$, $\varphi \in
C_c (G)$,  is mean almost periodic and the equivalence between (i)
and (ii) follows.

(iii)$\Longrightarrow$(ii): This follows easily as any $N_\varphi$,
$\varphi \in C_c (G)$, is a continuous function on $\Omega$.

(ii)$\Longrightarrow$(iii): It is not hard to see that the set of
$N_\varphi$, $\varphi \in C_c (G)$, separates the points of $\Omega$
and is closed under complex conjugation.  Hence, the algebra
generated by the $N_\varphi$ is dense in the continuous functions on
$\Omega$ with respect to the supremum norm (see e.g. \cite{BL} for
further discussion of this type of argument). This gives that (ii)
implies (iii).
\end{proof}

When dealing with TMDS $(\Omega,G,m)$  we can now use
the previous proposition  to replace the assumption that $\omega\in
\Omega$ is mean almost periodic as element of the dynamical system
$(\Omega,G,m)$ by the assumption that $\omega$ is a mean almost
periodic measure (and, similarly,  with mean almost periodic
replaced by Besicovitch almost periodic, Weyl almost periodic, weak
almost periodic, Bohr almost periodic). This allows then for
reformulations of our main results. We only state one of the
applications and leave the remaining ones to the reader.

\begin{theorem} Let $(\varOmega,G,m)$ be an ergodic  TMDS. Assume that $G$ is
second countable and that $(B_n)$ is a F\o lner sequence along which
Birkhoff's ergodic theorem holds.  Then, the following assertions
are equivalent:
\begin{itemize}
\item[(i)] $(\varOmega,G,m)$ has pure point spectrum.

\item[(ii)] Almost every $\omega\in \varOmega$ is a mean almost periodic measure.

\end{itemize}

\begin{remark} The previous theorem  was first shown in \cite{LeSpStr}. Here,
we have provided a different proof.
\end{remark}
\end{theorem}

\section{Abstract generalisations}\label{s:abstract}
When closing the article, it may be instructive to stop a moment to
have a look at the overall theme of this article from a more
abstract point of view. The general approach in this article may be
described as follows: We consider a dynamical system $(X,G)$ and say
that a point $x\in X$ is $(*)$ almost periodic if $\cA_x$ consists
of $(*)$ almost periodic functions only. Then, the preceding
sections have been devoted to a thorough study of consequences of
existence of $(*)$ almost periodic points for $(*)$ being replaced
by Bohr, weak, Weyl, Besicovitch and Mean (and in this order these
are increasingly weaker notions of almost periodicity).  Now, of
course, any other concept of almost periodicity for a function could
also be taken as the starting point of the theory. Then, some of our
considerations will easily carry over. This is discussed in this
section.

\smallskip

Let  $\Cu (G)$ be the set of  uniformly continuous bounded functions
on $G$ and consider a dynamical system $(X,G)$. Whenever  we are
given an $\AAA \subset \Cu(G)$ we can define for $p \in X$
\begin{displaymath}
\AAA_p:=\{ f \in C(X):f_p \in \AAA \}
\end{displaymath}
Then, $\AAA_p$ inherits various properties of $\AAA$. In particular,
if $\AAA$ is an algebra, then so is $\AAA_p$ and if $1 \in \AAA$
then $\AAA_p$ contains the constant function $1$. Moreover, if
$\AAA$ is closed in $(\Cu(G), \| \cdot \|_\infty)$ then $\AAA_p$ is
closed in $(C(X), \| \cdot \|_\infty)$. Then, as abstraction of Lemma
\ref{l:char-ap-functions} (with the same proof), we obtain the
following lemma.

\begin{lemma}\label{l:abstract-stone-w}
Let $\AAA$ be a closed subalgebra of $\Cu(G)$ containing the
constant function 1. Then, the following assertions are equivalent for $p\in
X$:
\begin{itemize}
\item[(i)] $\AAA_p \subseteq \AAA$.
  \item [(ii)] $\AAA_p =C(X)$.
  \item [(iii)] $\AAA_p$ separates the points of $\overline{Gp}$.
\item [(iv)] For each $s \in G$ the function $d^{(s)} \in C(X)$
with $d^{(s)} (y):= d(sp,y)$ belongs to $\AAA_p$.
\end{itemize}
\end{lemma}

A particular way to obtain a closed algebra $\AAA$ is by suitable
seminorms. This is discussed next: Call a seminorm $N$ on $\Cu (G)$
\textit{admissible} if it is $G$-invariant and satisfies
\begin{itemize}
\item $N(f) \leq N(g)$ whenever $|f|\leq g$
\item $N(1) = 1$
\end{itemize}
Note that any admissible seminorm $N$ satisfies $N\leq
\|\cdot\|_\infty$  (as  $|f|\leq \|f\|_\infty \cdot 1$).

\begin{remark}[Examples] It is not hard to see that
 $\|\cdot\|_\infty$ and $ \overline{M}\circ |\cdot|$ as well as
  \begin{displaymath}
N(f):= \limsup_{n\to\infty} \sup_{t \in G} \frac{1}{|B_n|} \int_{t+B_n} \left|f(t) \right| \dd t
  \end{displaymath}
are admissible seminorms on $C_u (G)$.
\end{remark}

\begin{definition} (a) We say that an $f\in C_u (G)$ is \textit{$N$-almost periodic} if for
any $\varepsilon >0$ the set
$$\{t\in G: N(|f - f(\cdot - t)|) <\varepsilon\}$$
is relatively dense in $G$.

(b) We say that an $f\in C_u (G)$ is \textit{$N$-trig almost
periodic} if for any $\varepsilon >0$, there exists some
trigonometric polynomial $P$ such that $N(|f-P|) <\varepsilon$.
\end{definition}

\begin{lemma} Let $N$ be an admissible seminorm. Then,
\[
\AAA^N:= \{ f \in C_u(G) : f \mbox{ is } N -\mbox{almost periodic } \}
\]
and
\[
\AAA^T:= \{ f \in C_u(G) : f \mbox{ is } N -\mbox{trig almost periodic } \}
\]
are closed subalgebras of $\Cu(G)$.  Both subalgebras contain the
Bohr almost periodic functions and  $\AAA^T \subseteq \AAA^N$ holds.

In particular $1 \in \AAA^N$.
\end{lemma}
\begin{proof} The statement on   $\AAA^N$ is an  abstraction of
Theorem \ref{t:map-algebra} in our context. It can be shown by
replacing  $M\circ |\cdot|$ with $N$ in the proof of this theorem.
Similarly, the statement on $\AAA^T$ is an abstraction of Lemma
\ref{l:bap-algebra} in our context and can be shown by mimicking the
proof of that lemma. This also gives the last statement.
\end{proof}

\begin{definition}[$N$-almost periodic points]
Let $(X,G)$ be a dynamical system and $N$ an admissible seminorm on
$C_u (G)$. We say, that an $x\in X$ is \textit{ $N$-almost periodic}
if $f_x$ is $N$-almost periodic for any $f\in C(X)$, i.e. if
$\AAA^N_x = C(X)$ holds.
\end{definition}

For a continuous metric $d$ on $X$ and $x\in X$ we define the
$$
d^{N,x}: G\longrightarrow
[0,\infty),\qquad d^{N,x} (t) := N((s\mapsto d(sx, (t+s) x) \,.
$$
Reasoning
as in Section \ref{s:main}, we can infer that for any continuous
metric $d$ the function $d^{N,x}$  is uniformly continuous. Having
set up things, we can now discuss the following abstraction of the
results in  Section \ref{s:main} (where we include some details for
the convenience of the reader).

Analogously to Lemma \ref{l:map-via-bohr} we find the following.

\begin{lemma}\label{l:periods-vs-periodicity} Let $(X,G)$ be a dynamical system.
Let $N$ be admissible.  Let $d$ be a continuous metric on $X$.
 Then, the following assertions for $x\in X$ are equivalent:
\begin{itemize}
\item[(i)]  For any $\varepsilon >0$ the  set $$\{t\in G : d^{N,x} (t) < \varepsilon\}$$
is relatively dense.
\item[(ii)] The function $d^{N,x}$ is Bohr almost periodic.
\end{itemize}
\end{lemma}

The following is an abstraction of  both the equivalence between (i)
and (ii) in  Lemma \ref{l:char-ap-functions} (with $N = \overline{M}
\circ |\cdot|$) and Theorem \ref{t:Bohr} (with
$N=\|\cdot\|_\infty$).

\begin{theorem}\label{t:abstract-nonsense-main} Let $(X,G)$ be a dynamical system.
Let $N$ be admissible.  Then, for $x\in
X$ the following assertions are equivalent:
\begin{itemize}
\item[(i)] The point  $x$ is $N$-almost periodic.
\item[(ii)] There exists a continuous metric $d$ on $X$
such that $d^{N,x}$ is Bohr almost periodic.
\item[(iii)] For every  continuous metric $d$ on $X$ the function
$d^{N,x}$ is Bohr almost periodic.
\end{itemize}
\end{theorem}

\begin{remark} As $X$ is a compact metric space, a
metric on $X$ is continuous if and only if it generates the
topology.
\end{remark}

\begin{proof} (iii)$\Longrightarrow$(ii): This is clear.

\smallskip
(ii)$\Longrightarrow$(iii): This is the analogue of Lemma
\ref{l:independence} in our context. It can be shown by a variant of
the proof of that lemma. Some extra effort is needed as $N$ is not
defined on measurable bounded functions but only on $\Cu (G)$. This
is tackled by means of Urysohn lemma. As every locally compact group
is a normal space this lemma allows one to separate arbitrary closed
disjoint sets by continuous functions and this is what we will use.
Here are the details:

Let $e$ be any metric on $X$ such that $e^{N,x}$ is almost periodic
and let $d$ be any other metric on $X$, such that $d,e$ generate the
topology. Without loss of generality we can assume that $d, e \leq
1$.  Let $\varepsilon>0$ be arbitrary.  By Lemma
\ref{l:periods-vs-periodicity} it suffices to show that the set of
$t\in G$ with $d^{N,x}(t)\leq \varepsilon$ is relatively dense.

Choose $\delta' >0$ with $d(y,z) < \frac{\varepsilon}{2}$ whenever
$e(y,z) <\delta'$. Set
\[
\delta:=\frac{\varepsilon}{4 \delta' } \,.
\]
Let $t\in G$ be so that $e^{N,x}(t)< \delta$. By Lemma
\ref{l:periods-vs-periodicity} the set of such $t\in G$ is
relatively dense. Thus,  it remains to show $d^{N,x}(t)<\varepsilon$
for any such $t\in G$. Define $e_{t,x}$ and $d_{t,x}$ on $G$ by
$e_{t,x}(s):=e(sx,(t+s)x)$ and $d_{t,x}(s):=d(sx,(t+s)x)$.
 Set
\begin{displaymath}
A:=\{ s : e_{t,x} \geq \delta' \} \,;\qquad B:=\{ s : e_{t,x} \leq \frac{\delta'}{2} \}  \,.
\end{displaymath}
Then, by Urysohn's Lemma, there exists some continuous function $f: G
\to [0,1]$ such that $f(x)=1$ for all $x\in A$ and $f(x)=0$ for all
$x \in B$. Set $g:=1-f$. Then, $ e_{t,x} \geq  \frac{\delta'}{2} f$
and hence, $
 \frac{\delta'}{2} N(f) \leq N(e_{t,x}) \leq \delta$,
showing
$$N(f) \leq \frac{2\delta}{\delta'}=\frac{\varepsilon}{2}.$$
Moreover,  $d_{t,x}(s)g(s) \leq \frac{\varepsilon}{2}$. Indeed, if
$s \in A$ then  $g(s)=0$ by definition of $g$  and if $s \notin A$
then $ d_{t,x}(s) \leq \frac{\varepsilon}{2}$ by our choice of
$\delta'$. This shows
\[
N(d_{t,x}  ) \leq \frac{\varepsilon}{2} \,.
\]
Therefore, using $d \leq 1$ we obtain
$$
d^{N,x}(t) = N(d_{t,x}) \leq  N(d_{t,x}f)+N(d_{t,x}g)
   \leq N(f)+N(d_{t,x}g) \leq \frac{\varepsilon}{2} +
   \frac{\varepsilon}{2} = \varepsilon.
$$

\smallskip

(i)$\Longrightarrow$(ii):  This can  be shown as
(ii)$\Longrightarrow$(i) in Lemma \ref{l:char-ap-functions}.

\smallskip

(ii)$\Longrightarrow$(i): As in the  proof of
(i)$\Longrightarrow$(iv) of Lemma \ref{l:char-ap-functions} we infer
from (ii) that $d^{z}_x$ is $N$-almost periodic for any $z\in X$.
Now, (i) follows from Lemma \ref{l:abstract-stone-w}.
\end{proof}

\begin{coro} Let $(X,G,m)$ be a dynamical system with metric $d$.
Let $N$ be an admissible seminorm and $x\in X$ be $N$-almost
periodic. If
$$
\int_X f(y)\, \dd m(y) \leq N(f_x)
$$
holds for all $f\in C(X)$,  then $(X,G,m)$ has pure point spectrum.
\end{coro}
\begin{remark} Note that the assumption holds whenever there exists
a F \o lner sequence $(B_n)$ along which Birkhoff's ergodic
theorem holds and $N$ satisfies $\overline{M} \circ |\cdot| \leq N$
and $x$ is generic with respect to $m$.
\end{remark}
\begin{proof}
As $x$ is  $N$-almost periodic the preceding
theorem (combined with Lemma \ref{l:periods-vs-periodicity}) gives
that for any $\varepsilon >0$ the set of $t\in G$ with $
d^{N,x}(t)<\varepsilon$ is relatively dense. By using the assumption
on $N$  with $f(y) = d(y,ty)$ (for $t\in G$ fixed), we furthermore
find  for the function
\[
\underline{d}: G\longrightarrow
[0,\infty),\qquad \underline{d}(t) = \int_X d(y,ty)\, \dd m(y),
\]
the inequality
$$\underline{d}(t) \leq d^{N,x} (t)$$
for all $t\in G$. Hence, for any $\varepsilon >0$ the set of $t\in
G$ with $\underline{d}(t) <\varepsilon$ is relatively dense as well.
This gives that $\underline{d}$ is Bohr almost periodic and the
desired statement on pure point spectrum follows from the main
result of  \cite{Len} (compare proof of Theorem \ref{t:main} as
well).
\end{proof}

\begin{appendix}

\section{Bohr almost periodic functions}\label{s:a:bap}
In this section, we briefly present some background on Bohr almost
periodic function. This is completely standard and can be found in
many places. We include a discussion here in order to be
self-contained and to set the perspective on the weaker (and less
known) notions of almost periodicity underlying our considerations.
For more details we refer the reader to \cite{Cord,MoSt}.

\medskip

Let $G$ be a locally compact abelian group. Recall from
Section~\ref{s:mean} that a continuous function $f :
G\longrightarrow \CC$ is called \textit{Bohr almost periodic} if for
any $\varepsilon >0$ the set of $t\in G$ with
\begin{equation}
\|f - f(\cdot -t) \|_\infty <\varepsilon \tag{$\spadesuit$}
\end{equation}
is relatively dense. Any such $t\in G$ is then called an
\textit{$\varepsilon$-almost period} of $f$. It turns out that a
continuous $f : G\longrightarrow \CC$ is Bohr almost
periodic if and only if $\overline{\{f(\cdot - t) : t\in G\}}$ is
compact, where the closure is taken with respect to the supremum
norm. It is not hard to see that any Bohr almost periodic function
is bounded and uniformly continuous. Moreover, the Bohr almost periodic
functions form a closed subalgebra of the algebra of all uniformly
continuous bounded functions on $G$. The main structural result on
Bohr almost periodic functions is that a function $f$ is Bohr almost
periodic if and only if for any $\varepsilon >0$ there exist
$k\in\NN$,  $\xi_1,\ldots, \xi_k\in \widehat{G}$ and $c_1, \ldots,
c_k\in \CC$ with

\begin{equation*}
\| f - \sum_{j=1}^k c_j \xi_j\|_\infty  < \varepsilon.   \tag{$\heartsuit$}
\end{equation*}
The basic idea of the weaker concepts of almost periodicity
discussed in the subsequent sections is to replace the supremum norm
$\|\cdot\|_\infty$ in $(\spadesuit)$ and $(\heartsuit)$ by suitable
(semi)norms arising by  averaging procedures.

\section{Mean almost periodic functions}\label{s:a:map}
The main result of this appendix shows that the bounded uniformly
continuous  mean almost periodic functions form a closed subalgebra
of the algebra of bounded uniformly continuous functions on $G$.
This is certainly well-known and a proof can be given by standard
means. For the convenience of the reader and in order to keep this
article self contained we include a discussion below. As our article
deals with abelian groups we assume that the group $G$ below is
abelian. Note, however, that this is not used in the proofs.

\medskip

We start with a general result on discrete geometry of groups.

\begin{prop}\label{prop B1}
Let $G$ be a locally compact abelian group. Let $D$ and $E$ be
relatively dense subsets of $G$ and $V$ a relatively compact open
neighborhood of the neutral element. Then, $ ((D-D) + V) \cap ((E -
E) + V) $ is relatively dense.
\end{prop}

\begin{proof} As both $D$ and $E$ are relatively dense, we can choose
an open relatively compact set $U\subset G$ with the property that
any translate of $U$ intersects both $D$ and $E$. As addition is
continuous on $G$, we can choose furthermore a relatively compact
open neighborhood $W$ with $W = -W$ and $W + W \subset V$. As $U$ is
relatively compact, there exist $N\in \NN$ and $z_1,\ldots z_N\in G$
with
$$U = \bigcup_{n=1}^N (z_n + W) \cap U.$$

Consider now an arbitrary $t\in D$. Then, there exists an $s\in E$
with $t - s \in U$. Hence, for any $t\in D$ we can find $s_t \in E$
and $n_t\in \{1,\ldots, N\}$ with
$$(t - s_t) \in z_{n_t} + W.$$
Fix now for each $n\in \{1,\ldots, N\}$  elements $t_n\in D$ and
$s_n\in E$ with
$$ (t_n - s_n) \in z_n + W.$$
(If some of the $n$ does not admit elements, we just remove this $n$
from our list.) Then, for any $t\in D$ we can find $s_t\in E$ and
$n\in \{1,\ldots, N\}$ such that both $ t- s_t$ and $ t_n - s_n$
belong to $z_n + W$. Hence, we find
$$t - s_t + w = t_n - s_n + w'$$
for suitable $w,w'\in W$. This gives
$$ t- t_n = s_t - s_n + v$$
with $v = w' - w \in W - W \subset V$. Now, we clearly have $t- t_n
\in (D-D) + V$ and $(s_t - s_n) + v\in (E-E) + V.$ Moreover, the set
of all $t - t_n$ is relatively dense as $t$ is an  arbitrary element
of the relatively dense $D$ and there are only finitely many $t_n$.
This finishes the proof.
\end{proof}

Let $G$ be a { $\sigma$-compact} locally compact abelian group and
$(B_n)$ a F\o lner sequence on $G$. Let $f$ be a uniformly
continuous bounded function on $G$. Let $\varepsilon >0$ be given.
As usual we say that a $t\in G$ is an $\varepsilon$-almost period of
$f$ if
$$\overline{M} ( |f- f(\cdot-t)|) <\varepsilon.$$
Denote the set of all $\varepsilon$-almost periods of $f$ by $
\mbox{AP}(f,\varepsilon)$. Then, it is not hard to see that
$$\mbox{AP}(f,\varepsilon) -\mbox{AP}(f,\varepsilon) \subset
\mbox{AP}(f,2 \varepsilon).$$  A uniformly continuous bounded $f :
G\longrightarrow \CC$ is mean almost periodic if for any
$\varepsilon >0$ the set $\mbox{AP}(f,\varepsilon)$ is relatively
dense.

\begin{lemma}\label{l:intersection-periods}
Let $G$ be a { $\sigma$-compact} locally compact abelian group and
$(B_n)$ a F\o lner sequence on $G$. Let a natural number $n$ and
uniformly continuous bounded mean almost periodic functions $f_1,
\ldots, f_n$ on $G$ be given. Then, the set
$$\bigcap_{k=1}^n \operatorname{AP}(f_k,\varepsilon)$$
is relatively dense in $G$ for any $\varepsilon >0$.
\end{lemma}
\begin{proof} This will be shown by induction in $n$. The case $n =1$ is
clear. So, assume now the statement holds for a chosen $n$. Let
$\varepsilon >0$ and uniformly continuous functions $f_1,\ldots,
f_{n+1}$ be given. Then, the set $D:=\bigcap_{k=1}^n \mbox{AP}(f_k,
\varepsilon / 3)$ is relatively dense by assumption. As the
functions $f_k$, $k=1,\ldots, n+1$,  are uniformly continuous we can
find an open relatively compact neighborhood $V$ of the neutral
element such that
$$\|f_k - f_k(\cdot-s)\|_\infty <
\frac{\varepsilon}{3}$$ for all $s\in V$ and $k=1,\ldots, n+1$.  Set
$E:=\mbox{AP}(f_{n+1},\varepsilon/3)$. Then, the previous
proposition gives that
$$((D- D) + V) \cap ((E-E) + V)$$
is relatively dense. On the other hand it is not hard to see that
$$(D-D)  + V\subset \mbox{AP}(f_k,\varepsilon), k=1,\ldots, n  \mbox{ and } (E-E) +
V\subset \mbox{AP}(f_{n+1}, \varepsilon).$$ This finishes the proof.
\end{proof}

\begin{theorem}\label{t:map-algebra} Let $G$ be a { $\sigma$-compact} locally compact abelian group and $(B_n)$
a F\o lner sequence on $G$. Then, the set of all uniformly
continuous bounded mean almost periodic functions is invariant under
taking complex conjugates and a  closed subalgebra of the uniformly
continuous bounded functions on $G$ equipped with
$\|\cdot\|_\infty$.
\end{theorem}
\begin{proof} We have to show that the set in question is closed
under complex conjugation,  sums, products, multiplication by
scalars and uniform convergence.

\smallskip

It is not hard to see that the set in question is closed under
complex conjugation,  uniform convergence and multiplication by
scalars.

\smallskip

We next show that it is closed under sums: Let $f,g$ be mean almost
periodic uniformly continuous bounded functions. Then, the previous
lemma easily gives that $f + g$ is also mean almost periodic.

\smallskip

Finally, we deal with products: Let $f,g$ be mean almost periodic
uniformly continuous bounded functions. Then, a short computation
gives for any $t\in G$
\begin{align*}
|f(s) g(s) - f(s-t) g(s-t)|
    &\leq |f(s) g(s) - f(s) g(s-t)| +|f(s)g(s-t) - f(s-t) g(s-t)|\\
    &\leq \|f\|_\infty |g(s) - g(s-t)|+ \|g\|_\infty |f(s) - f(s-t)|
\end{align*}
From this we easily obtain
$$\overline{M}(|fg - (fg)(\cdot-t)|) \leq \|f\|_\infty
\overline{M}(|g - g(\cdot-t)|) + \|g\|_\infty \overline{M}(|f  -
f(\cdot-t)|).$$ Now, the desired statement follows easily from the
preceding lemma.

\smallskip

The last statement is clear.
\end{proof}

For later use we also note the following proposition.

\begin{prop}\label{p:sum} Let $(f_n)$ be a sequence of uniformly
continuous bounded mean almost periodic functions on $G$ with
$\|f_n\|\leq 1$ for all $n$. Let $c_n>0$ with $\sum_{n=1}^{\infty} c_n <\infty$
be given. Then, there exists for any $\varepsilon >0$ a relatively
dense set $D$ in $G$ with
$$\overline{M}(\sum_{n=1}^{\infty} c_n |f - f(\cdot -t)|) <\varepsilon$$
for all $t\in D$.
\end{prop}
\begin{proof} Choose $n_0$ large enough so that $\sum_{k=n_0 +1}^\infty
c_k < \varepsilon /4$. Then, $\sum_{n=n_0 +1}^\infty c_n \|f -
f(\cdot -t)\|_\infty < \varepsilon /2 $ by assumption on the $f_n$
and the $c_n$. By Lemma \ref{l:intersection-periods} there exists a
relatively dense set $D$ in $G$ with $\overline{M}( |f_k - f_k
(\cdot -t)|) <\frac{\varepsilon}{2 n_0}$ for $k = 1,\ldots, n_0$.
Now, the statement follows easily.
\end{proof}

\begin{remark} The considerations of this section carry over when
$\overline{M}\circ |\cdot| $ is replaced by any invariant seminorm
$N$ on the algebra of bounded uniformly continuous functions on $G$
satisfying
\begin{itemize}
\item $N(f) \leq N(g)$ whenever $|f|\leq g$,
\item $N(1) =1$.
\end{itemize}
This point is taken up in Section \ref{s:abstract}.
\end{remark}

\section{Besicovitch almost periodic
functions and existence of means}\label{s:a:beap} We consider a {
$\sigma$-compact} locally compact abelian group $G$ together with a
F\o lner sequence $(B_n)$. Our aim is to study the set of uniformly
continuous bounded functions $f : G\longrightarrow \CC$, which are
Besicovitch almost periodic i.e. satisfy that for any $\varepsilon
>0$ there exist $k\in \NN$, $\xi_1,\ldots, \xi_k\in
\widehat{G}$ and $c_1,\ldots ,c_k\in\CC$ with
$$\overline{M} (|f - \sum_{j=1}^k c_{\xi_j} \xi_j|)< \varepsilon.$$
For further details and a more in depth study  of these functions we
refer to \cite{LeSpStr}.

\medskip

\begin{prop} Let $G$ be a $\sigma$-compact locally compact abelian group and $(B_n)$
a F\o lner sequence on $G$. Then, any uniformly continuous bounded
Besicovitch almost periodic function is mean almost periodic.
\end{prop}
\begin{proof} Let $\varepsilon >0$ be given. Choose $c_1,\ldots,
c_k\in \CC $ and $\xi_1,\ldots, \xi_k\in\widehat{G}$ such that
$P:=\sum_{j=1}^k c_j \xi_j$ satisfies
$$\overline{M}(| f - P|) <\varepsilon.$$
As $\overline{M}$ is invariant this inequality will then continue to
hold if $f$ is replaced by $f(\cdot -t)$ and $P$ is replaced by
$P(\cdot  - t)$ for any $t\in G$. Now, clearly $P$ is Bohr almost
periodic. Hence, there exists a relatively dense set $R\subset G$
with $\|P - P(\cdot - t)\|_\infty <\varepsilon$ for all $t\in R$.
This easily gives
$$\overline{M}( |f -f(\cdot -t)|) \leq \overline{M}(|f - P|) +
\overline{M}( |P - P(\cdot - t)|) + \overline{M}(|P(\cdot -t) -
f(\cdot - t)|) < 3 \varepsilon$$ for all $t\in R$.
\end{proof}

From the definition and simple algebraic manipulations we infer the
following.

\begin{lemma}\label{l:bap-algebra}
Let $G$ be a { $\sigma$-compact} locally compact abelian group and
$(B_n)$ a F\o lner sequence on $G$. Then, the set of all uniformly
continuous bounded Besicovitch almost periodic functions is
invariant under taking complex conjugates and a  closed subalgebra
of the uniformly continuous bounded functions on $G$ equipped with
$\|\cdot\|_\infty$. It  contains all Bohr almost periodic functions
and is contained in the algebra of mean almost periodic functions.
\end{lemma}
\begin{proof} The set is clearly  closed under taking
limits with respect to $\|\cdot\|_\infty$ as well as under  addition
and taking complex conjugates. To show that it is closed under
multiplication consider $f,g$ in this set and let $\epsilon
>0$ be arbitrary. Let $P,Q$ be trigonometric polynomials so that
$$
\overline{M}(|f-P|) < \frac{\varepsilon}{2\| g \|_\infty +4} \qquad
\mbox{ and } \qquad \overline{M}(|g-Q|)< \frac{\varepsilon}{4\| f
\|_\infty +1} \,.$$

Define
\begin{displaymath}
Q'(x):= \left\{
\begin{array}{lc}
Q(x) & \mbox{  if } |Q(x)| \leq \| g \|_\infty +1 \\
\frac{Q(x)}{|Q(x)|}\left(  \| g \|_\infty +1  \right) & \text{otherwise}
\end{array} \,.
\right.
\end{displaymath}
Then, $Q'$ is a Bohr almost periodic function and $|g-Q'| \leq
|g-Q|$, which gives $\overline{M}(|g-Q'|) \leq \overline{M}(|g-Q|)$.

Since $Q'$ is Bohr almost periodic, there exists a trigonometric
polynomial $R$ such that $\| Q'-R \|_\infty < \min\{
\frac{\varepsilon}{4\| f \|_\infty +1}, 1 \}$. In particular,
\[
\| R \|_\infty \leq \|Q'\|_\infty+1 \leq  \| g \|_\infty +2 \,.
\]
Then,
\begin{align*}
  \overline{M}(|fg-PR|) &\leq \overline{M}(|fg-fR|)+ \overline{M}(|fR-PR|) \\
  &\leq \| f \|_\infty \overline{M}(|g-R|)+\| R \|_\infty
  \overline{M}(|f-P|)\\
   &\leq \| f \|_\infty \left( \overline{M}(|g-Q'|)+\overline{M}(|Q'-R|) \right)
   + \frac{\varepsilon}{2} \\
  &\leq \| f \|_\infty \left( \overline{M}(|g-Q'|)+\| Q'-R\|_\infty \right)+
  \frac{\varepsilon}{2}\\
  & < \varepsilon \,.
\end{align*}
This shows that the set in question is closed under multiplication.
It clearly contains all Bohr almost periodic functions and is
contained in the set of mean almost periodic functions.
\end{proof}

\begin{remark}
The functions referred to as Besicovitch almost periodic were
introduced by Besicovitch in \cite{Bes} for $G =\RR$. A
corresponding class of functions was then studied by F\o lner for
general locally compact abelian groups \cite{Fol}. This class, however, does
not coincide with  the Besicovitch class for $G =\RR$. Another
approach to Besicovitch almost periodic functions is developed by
Davis in \cite{Dav}.  An  account of these developments with a focus
on aperiodic order is given in Lagarias survey \cite{Lag}. Here, we
have taken a `shortcut': We have not defined Besicovitch almost
periodic functions by some intrinsic features. Instead we have
defined them by what would be a main result in a proper theory
starting with an intrinsic definition.
\end{remark}

A crucial feature of Besicovitch almost periodic function is
existence of means in the following sense.

\begin{lemma} Let $G$ be a {$\sigma$-compact} locally compact abelian group and $(B_n)$
a F\o lner sequence on $G$. Let $f: G\longrightarrow \CC$ be a
uniformly continuous, bounded Besicovitch almost periodic
function. Then, the limit
\[
\Average(f)=\lim_{n\to \infty}\frac{1}{|B_n|} \int_{B_n} f(s)\, \dd
s
\]
exists.
\end{lemma}
\begin{proof} The statement is well-known for functions of the form $f =
\sum_{j=1}^k c_j \xi_j$ with $\xi_1,\ldots, \xi_k\in \widehat{G}$,
$c_1,\ldots, c_k\in \CC$. It then follows by a limiting procedure
for  Besicovitch almost periodic functions.
\end{proof}

In fact, existence of means together with a Parseval type equality
is a characterizing feature of Besicovitch almost periodic
functions.

\begin{prop}\label{p:fourier-expansion}
Let $G$ be a { $\sigma$-compact} locally compact abelian group and
$(B_n)$ a F\o lner sequence on $G$. Let $f: G\longrightarrow \CC$ be
a uniformly continuous and bounded. Then, $f$  is Besicovitch almost
periodic if and only if there exists a countable set $F\subset
\widehat{G}$ such that the  following three statements hold:
\begin{itemize}
\item  The limit  $\Average(|f|^2) = \lim_{n\to \infty} \frac{1}{|B_n|}
\int_{B_n} |f(t)|^2 \, \dd t$ exists.
\item For any  $\xi\in F$ the limit
\[
\Average(f \overline{\xi})  =\lim_{n\to \infty} \frac{1}{|B_n|}
\int_{B_n} f(t)\, \overline{\xi(t)} \, \dd t
\]
exists.

\item The equality
\[
\Average(|f|^2)  =\sum_{\xi \in F} |\Average(f \overline{\xi})|^2
\]
holds.
\end{itemize}
In this case $\Average(f \overline{\xi})= 0$ holds for all $\xi \in
\widehat{G}\setminus F$.
\end{prop}
\begin{proof} Let $f$ be Besicovitch almost periodic.  Clearly, any $\xi \in
\widehat{G}$ is Bohr almost periodic and hence Besicovitch almost
periodic. Thus, $f\overline{\xi}$ is Besicovitch almost periodic as
product of Besicovitch almost periodic functions. Hence, the second
point holds (even for all $\xi \in\widehat{G}$). Similarly, $|f|^2$
is Besicovitch almost periodic. Given this, the first point follows.
The last point is contained in \cite{LeSpStr}.

\smallskip

Now, consider $f$ satisfying the three points above.   Let
$\xi_1,\xi_2,\ldots $ be an enumeration of the $\xi \in F$. A
computation involving Cauchy--Schwarz' inequality in the first step,
$\overline{M}(1) =1$ in the second step  and existence of averages
$A$ in the third step gives
\begin{eqnarray*}
&\overline{M}&(|f -\sum_{j=1}^N A(f\overline{\xi_j} ) \xi_j|)^2 \\
&\leq &\overline{M}(|f -\sum_{j=1}^N A(f\overline{\xi_j} ) \xi_j|^2) \overline{M}(1)\\
&=&\overline{M}\left(|f|^2  - f \overline{\sum_{j=1}^N
A(f\overline{\xi_j} ) \xi_j} - \overline{f}\sum_{j=1}^N
A(f\overline{\xi_j} ) \xi_j+ \sum_{j,k=1}^N A(f\overline{\xi_j} )
\overline{ A(f\overline{\xi_k})} \xi_j \overline{\xi_k}\right)  \\
 &=& \Average(|f|^2) -\Average( f \overline{\sum_{j=1}^N
A(f\overline{\xi_j} ) \xi_j} ) - \Average( \overline{f}\sum_{j=1}^N
A(f\overline{\xi_j} ) \xi_j ) + \Average( \sum_{j=1,k}^N
A(f\overline{\xi_j})
\overline{A(f\overline{\xi_k})} \xi_j \overline{\xi_k})\\
&=& \Average(|f|^2) - \sum_{j=1}^N |A(f
\overline{\xi})|^2\\
&\to& 0
\end{eqnarray*}
for $N\to\infty$. Here, the penultimate step is a direct computation invoking that $A$
is linear with $A(\eta)=0$ for $0\neq \eta\in \widehat{G}$.
Indeed, this shows that $f$ can be approximated in
mean by a trigonometric polynomial arbitrarily well. This finishes
the proof.
\end{proof}

\begin{remark}
The considerations of this appendix easily carry over
to bounded measurable functions (instead of uniformly continuous
bounded functions). As the article is concerned with continuous
functions we do not elaborate on this but rather leave the details
to the reader.
\end{remark}

\section{Weyl almost periodic functions and uniform means}\label{s:a:wap}

In this appendix, we consider a uniform type of mean.

\medskip

Let $(B_n)$ be a F\o lner sequence in $G$ and define for bounded  $f
: G\longrightarrow \RR$ and $n\in\NN$
\[
\overline{M}_n (f) := \sup_{r\in G} \frac{1}{|B_n|} \int_{B_n + r}
f(t)\, \dd t.
\]

\begin{prop}\label{p:a:final-averaging} For a bounded measurable function
$f:G\longrightarrow \RR$ and $\varepsilon >0$ the following
assertions are equivalent:

\begin{itemize}

\item[(i)] There exists an $N\in \NN$ with $\overline{M}_N (f) < \varepsilon$.

\item[(ii)] There exists an $N_0\in\NN$ with $\overline{M}_N (f) <\varepsilon$
for all $N\geq N_0$.

\end{itemize}
\end{prop}

\begin{proof} (ii)$\Longrightarrow$(i): This is clear.

\smallskip

(i)$\Longrightarrow$(ii): Consider  $n\in\NN$. Define for $r\in G$
\[
C_{r}:= \frac{1}{|B_N|}\int_{B_N }
\left(\int\frac{1}{|B_n|}\int_{B_n + r} f(s +t)\, \dd s \right)\,
\dd t.
\]
Fubinis theorem and the assumption (i) directly give
\[
C_{r}=\frac{1}{|B_n|}\int_{B_n + r} \left(\frac{1}{|B_N|}\int_{B_N }
f(s +t)\, \dd t \right)\, \dd s \leq \frac{1}{|B_n|}\int_{B_n + r}
\overline{M}_N (f)\, \dd s = \overline{M}_N (f).
\]

On the other hand we can easily see that (for large $n$) the
additional averaging over $B_N$ does not play a role. More
specifically, we can compute as follows:
\begin{eqnarray*}
\left| C_{r} - \frac{1}{|B_n|} \int_{B_n + r} f(s)\, \dd s \right|
 &=& \left| C_{r} - \frac{1}{|B_N|} \int_{B_N } \left(\frac{1}{|B_n|} \int_{B_n + r} f(s)\, \dd s\right)\, \dd t\right|\\
 &=& \left|\frac{1}{|B_N|}\int_{B_N }
\left(\frac{1}{|B_n|}\int_{B_n + r} (f(s +t) - f(s) )\, \dd s
\right)\,
\dd t\right|\\
&\leq & \frac{1}{|B_N|} \int_{B_N} 2\|f\|_\infty \frac{|(B_n + t)\,
\triangle\, B_n| }{|B_n| }\, \dd t.
\end{eqnarray*}
Now, due to the F\o lner condition the integrand in the last term
can easily be seen to go pointwise to $0$  for $n\to \infty$. As
$B_N$ is compact, we find convergence to zero of the integral for
each fixed $N$. As there is no dependence on $r\in G$ this
convergence is independent of $r\in G$. This easily gives the
desired statement.
\end{proof}

A bounded $f : G\longrightarrow \CC$ is Weyl almost periodic if
for all $\varepsilon >0$ there exist   $k\in \NN$, $\xi_1,\ldots,
\xi_k\in \widehat{G}$ and $c_1,\ldots ,c_k\in\CC$ with
$$\limsup_{n\to\infty} \overline{M}_n (|f - \sum_{j=1}^k c_{\xi_j} \xi_j|) <
\varepsilon.$$ By the preceding proposition it is possible to
replace this condition by the requirement that for each $\varepsilon
>0$ there exists an $N\in \NN$ , $k\in \NN$, $\xi_1,\ldots, \xi_k \in \widehat{G}$
and $c_1,\ldots ,c_k\in\CC$ with
$$ \overline{M}_N (|f - \sum_{j=1}^k c_{\xi_j} \xi_j|) <
\varepsilon.$$

As usual it is also  possible to characterize this by relative
denseness of $\varepsilon$-almost periods. More specifically, as
discussed in \cite{Spi} a measurable bounded $f : G\longrightarrow
\CC$ is Weyl almost periodic if and only if  for each $\varepsilon
>0$ there exists a relatively dense set $D\subset G$ and an $N_0\in
\NN$ with
$$\overline{M}_n (|f - f(\cdot - t)|) < \varepsilon$$
for all $t\in D$ and $n\geq N_0$. By the preceding proposition
validity for all $n\geq N_0$ can be replaced by validity for one
$n$.

Clearly, the set of Weyl almost periodic functions forms an algebra,
is closed under uniform limits and under multiplication with
Bohr almost periodic functions. Moreover, a crucial feature of Weyl
mean almost periodic functions is uniform existence of means.

\begin{lemma}
Let $G$ be a locally compact abelian group and $(B_n)$
a F\o lner sequence on $G$. Let $f: G\longrightarrow \CC$ be a
bounded Weyl almost periodic function. Then, for any sequence
$(r_n)$ in $G$ the limit
\[
\lim_{n\to \infty}\frac{1}{|B_n|} \int_{B_n + r_n} f(s)\, \dd s
\]
exists and is independent of the sequence. In particular, the
convergence is uniform in the chosen sequence.
\end{lemma}
The \textit{proof} follows  the same lines as the proof of the
corresponding statement for Besicovitch almost periodic functions in
the preceding section. For this reason we leave it to the reader.
\end{appendix}

\end{document}